\documentclass[a4paper,oneside,11pt]{article}

\usepackage[a4paper]{geometry}
\usepackage{aeguill}                 
\usepackage{graphicx}
\usepackage{xcolor}
\usepackage{amsmath,amsfonts,amssymb,amsthm}
\usepackage{pstricks}
\usepackage{epstopdf}
\usepackage[utf8]{inputenc} 
\usepackage[T1]{fontenc} 
 
\usepackage[english]{babel} 

\title{Central limit theorems for mapping class groups and $\text{Out}(F_N)$}
\author{Camille Horbez}

\usepackage[all]{xy}
\usepackage{bbm}
\begin{document}
\maketitle
\newtheorem{de}{Definition} [section]
\newtheorem{theo}[de]{Theorem} 
\newtheorem{prop}[de]{Proposition}
\newtheorem{lemma}[de]{Lemma}
\newtheorem{cor}[de]{Corollary}
\newtheorem{propd}[de]{Proposition-Definition}

\theoremstyle{remark}
\newtheorem{rk}[de]{Remark}
\newtheorem{ex}[de]{Example}
\newtheorem{question}[de]{Question}

\normalsize

\newcommand{\coucou}[1]{\footnote{#1}\marginpar{$\leftarrow$}}

\begin{abstract}
We prove central limit theorems for the random walks on either the mapping class group of a closed, connected, orientable, hyperbolic surface, or on $\text{Out}(F_N)$, each time under a finite second moment condition on the measure (either with respect to the Teichmüller metric, or with respect to the Lipschitz metric on outer space). In the mapping class group case, this describes the spread of the hyperbolic length of a simple closed curve on the surface after applying a random product of mapping classes. In the case of $\text{Out}(F_N)$, this describes the spread of the length of primitive conjugacy classes in $F_N$ under random products of outer automorphisms. Both results are based on a general criterion for establishing a central limit theorem for the Busemann cocycle on the horoboundary of a metric space, applied to either the Teichmüller space of the surface, or to Culler--Vogtmann's outer space.
\end{abstract}

\setcounter{tocdepth}{3}
\tableofcontents

\section*{Introduction}

Central limit theorems in noncommutative settings have a long history, and have already been established in a lot of various different contexts. In the case of random products of matrices, a classical theorem of Furstenberg \cite{Fur63} asserts that if $(A_i)_{i\in\mathbb{N}}$ is a sequence of random matrices, all distributed with respect to some probability law $\mu$ on $GL(N,\mathbb{R})$ whose support generates a noncompact subgroup of $GL(N,\mathbb{R})$ that does not virtually preserve any proper linear subspace of $\mathbb{R}^{\mathbb{N}}$, then under a first moment assumption on $\mu$, there exists $\lambda>0$ such that for all $v\in\mathbb{R}^{N}\smallsetminus\{0\}$, almost surely, one has $$\lim_{n\to +\infty}\frac{1}{n}\log||A_n\dots A_1.v||=\lambda.$$ Central limit theorems in this context date back to the works of Furstenberg--Kesten \cite{FK60}, Le Page \cite{LP82}, Guivarc'h--Raugi \cite{GR85}, Goldsheid--Margulis \cite{GM89}. These assert, under some conditions on $\mu$, that the variables $$\frac{\log||A_n\dots A_1.v||-n\lambda}{\sqrt{n}}$$ converge in law toward a centered Gaussian law on $\mathbb{R}$, which does not depend on the  vector $v\in\mathbb{R}^N\smallsetminus\{0\}$. These central limit theorems were classically established under an exponential moment assumption on the measure $\mu$. Their proofs relied on establishing a spectral gap property for the transfer operator of the Markov chain corresponding to the random walk, in a well-chosen space of sufficiently regular functions. Recently, Benoist--Quint \cite{BQ13} gave a new approach to the central limit theorem on linear groups, which enabled them to relax the assumption made on the measure to a second moment condition.
\\
\\
\indent Central limit theorems have also been established for free groups (Sawyer--Steger \cite{SS87}) or more generally word-hyperbolic groups (Björklund \cite{Bjo10}), describing the spread of the word length of the element obtained at time $n$ of the random walk (again, the limiting law is Gaussian). Again, their proofs required exponential moment assumptions on the measure. Benoist--Quint's new approach also enabled them to similarly relax the moment assumption in this context \cite{BQ14}.
\\
\\
\indent The goal of the present paper is to prove central limit theorems on mapping class groups of closed, connected, orientable, hyperbolic surfaces, and on the group $\text{Out}(F_N)$ of outer automorphisms of a finitely generated free group. In the case of mapping class groups, we will establish a central limit theorem for the hyperbolic lengths of essential simple closed curves, under application of a random product of diffeomorphisms. In the case of $\text{Out}(F_N)$, we will establish a central limit theorem for the word lengths of primitive conjugacy classes of $F_N$, under application of a random product of outer automorphisms of $F_N$. 

\paragraph*{Central limit theorem on mapping class groups.} Let $S$ be a closed, connected, oriented, hyperbolic surface, and let $\rho$ be a hyperbolic metric on $S$. The mapping class group $\text{Mod}(S)$ is the group of isotopy classes of orientation-preserving diffeomorphisms of $S$. Karlsson established in \cite{Kar14} a version of the law of large numbers for the random walk on $\text{Mod}(S)$, estimating the typical growth of curves under random products of diffeomorphisms of the surface. Given a probability measure $\mu$ on $\text{Mod}(S)$, the \emph{(left) random walk} on $(\text{Mod}(S),\mu)$ is the Markov process whose position $\Phi_n$ at time $n$ is obtained by successive multiplications on the left of $n$ independent $\mu$-distributed increments $s_i$, i.e. $\Phi_n=s_n\dots s_1$. A probability measure on $\text{Mod}(S)$ is \emph{nonelementary} if the subsemigroup of $\text{Mod}(S)$ generated by the support of $\mu$ is a subgroup of $\text{Mod}(S)$ that contains two independent pseudo-Anosov mapping classes.

Karlsson proved that if $\mu$ is a nonelementary probability measure on $\text{Mod}(S)$ with finite first moment with respect to the Teichmüller metric, then there exists a (deterministic) real number $\lambda>0$ such that for all essential simple closed curves $c$ on $S$, and almost every sample path $(\Phi_n)_{n\in\mathbb{N}}$ of the random walk on $(\text{Mod}(S),\mu)$, one has $$\lim_{n\to +\infty}\frac{1}{n}\log l_{\rho}(\Phi_n(c))=\lambda.$$ Here $l_{\rho}(\Phi_n(c))$ is the smallest length of a curve in the isotopy class of $\Phi_n(c)$, measured by integration of the metric $\rho$. The growth rate $\lambda$ is also equal to the drift of the random walk on $(\text{Mod}(S),\mu)$ with respect to the Teichmüller metric. We will establish the following central limit theorem, under a second moment condition on $\mu$. In the statement, we denote by $\mu^{\ast n}$ the $n^{\text{th}}$ convolution of $\mu$.

\begin{theo}\label{intro}
Let $S$ be a closed, connected, oriented, hyperbolic surface, and let $\rho$ be a hyperbolic metric on $S$. Let $\mu$ be a nonelementary probability measure on $\text{Mod}(S)$ with finite second moment with respect to the Teichmüller metric. Let $\lambda>0$ be the drift of the random walk on $(\text{Mod}(S),\mu)$ with respect to the Teichmüller metric.\\ 
Then there exists a centered Gaussian law $N_{\mu}$ on $\mathbb{R}$ such that for every compactly supported continuous function $F$ on $\mathbb{R}$, and all essential simple closed curves $c$ on $S$, one has $$\lim_{n\to +\infty}\int_{\text{Mod}(S)} F \left(\frac{\log l_{\rho}(\Phi(c))-n\lambda}{\sqrt{n}}\right)d\mu^{\ast n}(\Phi) = \int_{\mathbb{R}}F(t)dN_{\mu}(t),$$ uniformly in $c$.
\end{theo}

\paragraph*{Central limit theorem on $\text{Out}(F_N)$.}

Let $N\ge 2$, let $F_N$ be a free group of rank $N$, and let $\text{Out}(F_N)$ denote its outer automorphism group. Let $\mu$ be a probability measure on $\text{Out}(F_N)$. We established in \cite{Hor14} the following analogue of Karlsson's theorem for the random walk on $\text{Out}(F_N)$, estimating the growth of nontrivial conjugacy classes in $F_N$ under application of a random product of outer automorphisms. Assume that the probability measure $\mu$ on $\text{Out}(F_N)$ is \emph{nonelementary}, i.e. the subgroup of $\text{Out}(F_N)$ generated by its support is not virtually cyclic, and does not virtually preserve the conjugacy class of any proper free factor of $F_N$, and assume that $\mu$ has finite first moment with respect to the asymmetric Lipschitz metric $d_{CV_N}$ on Culler--Vogtmann's outer space $CV_N$. Then there exists a Lyapunov exponent $\lambda>0$ such that for all primitive elements $g\in F_N$ and almost every sample path $(\Phi_n)_{n\in\mathbb{N}}$ of the left random walk on $(\text{Out}(F_N),\mu)$, one has $$\lim_{n\to +\infty}\frac{1}{n}\log||\Phi_n(g)||=\lambda,$$ where $||\Phi_n(g)||$ denotes the smallest word length of a conjugate of $\Phi_n(g)$, written in some prescribed free basis of $F_N$. Here, we recall that an element $g\in F_N$ is \emph{primitive} if it belongs to some free basis of $F_N$, and we denote by $\mathcal{P}_N$ the collection of all primitive elements of $F_N$. Again, the Lyapunov exponent $\lambda$ is also equal to the \emph{drift} of the random walk on $(\text{Out}(F_N),\mu)$, i.e. $$\lambda=\lim_{n\to +\infty}\frac{1}{n}d_{CV_N}(\Phi_n.o,o)$$ for almost every sample path $(\Phi_n)_{n\in\mathbb{N}}$ of the random walk. We will establish a central limit theorem for the variables $\log||\Phi_n(g)||$, under a second moment assumption on $\mu$.

\begin{theo}\label{tcl-intro}
Let $\mu$ be a nonelementary probability measure on $\text{Out}(F_N)$ with finite second moment with respect to $d_{CV_N}$. Let $\lambda>0$ be the drift of the random walk on $(\text{Out}(F_N),\mu)$ with respect to $d_{CV_N}$. Then there exists a centered Gaussian law $N_{\mu}$ on $\mathbb{R}$ such that for every compactly supported continuous function $F$ on $\mathbb{R}$, and all primitive elements $g\in \mathcal{P}_N$, one has $$\lim_{n\to +\infty}\int_{\text{Out}(F_N)} F \left(\frac{\log ||\Phi(g)||-n\lambda}{\sqrt{n}}\right)d\mu^{\ast n}(\Phi) = \int_{\mathbb{R}}F(t)dN_{\mu}(t),$$ uniformly in $g$.
\end{theo}

\paragraph*{Strategy of proofs.} The present paper was inspired by the new approach by Benoist--Quint to the central limit theorem for linear groups \cite{BQ13} and hyperbolic groups \cite{BQ14}. This relies on the study of various cocycles. Given a countable group $G$, a compact $G$-space $X$ and a continuous cocycle $\sigma:G\times X\to\mathbb{R}$, Benoist--Quint developed a method for proving a central limit theorem for the cocycle $\sigma$. This follows the so-called Gordin's method, and requires proving that $\sigma$ is \emph{centerable}, i.e. can be written as 
\begin{equation}\label{coho}
\sigma(g,x)=\sigma_0(g,x)+\psi(x)-\psi(gx)
\end{equation}
for all $(g,x)\in G\times X$, where $\psi$ is a bounded measurable function on $X$, and where there exists $\lambda\in\mathbb{R}$ such that $$\int_{G}\sigma_0(g,x)d\mu(g)=\lambda$$ for all $x\in X$. One is then left showing a central limit theorem for $\sigma_0$, which can be done by using a classical central limit theorem for martingales due to Brown \cite{Bro71}.

In order to prove a central limit theorem for the word length in a hyperbolic group $G$, Benoist--Quint applied in \cite{BQ14} the above strategy to the Busemann cocycle on the horofunction boundary of $G$. We recall that the \emph{horofunction compactification} of a proper geodesic metric $G$-space $(X,d_X)$ is defined as the closure of the image of the embedding 
\begin{displaymath}
\begin{array}{cccc}
\psi:& X&\to &\mathcal{C}(X)\\
& z & \mapsto &\{x\mapsto d_X(x,z)-d_X(o,z)\} 
\end{array}
\end{displaymath}
\noindent where $\mathcal{C}(X)$ is the space of continuous functions on $X$, equipped with the topology of uniform convergence on compact subsets of $X$, and $o\in X$ is a basepoint. The \emph{Busemann cocycle} is the continuous cocycle $\beta$ on $G\times\partial_hX$ (where $\partial_hX$ is the horoboundary of $X$) defined by letting $\beta(g,h):=h(g^{-1}.o)$ for all $(g,h)\in G\times\partial_hX$. 
\\
\\
\indent In the mapping class group context, we will establish a central limit theorem for the Busemann cocycle on the horoboundary of the Teichmüller space $\mathcal{T}(S)$ of the surface, which we equip with the Teichmüller metric. It turns out that this is enough for proving Theorem \ref{intro}, since the Busemann cocycle on $\text{Mod}(S)\times\partial_h\mathcal{T}(S)$ is closely related to lengths of simple closed curves on $S$ (the metric on $\mathcal{T}(S)$ can indeed be controlled using lengths of curves). Similarly, in the case of $\text{Out}(F_N)$, it will be enough to prove a central limit theorem for the Busemann cocycle on the horoboundary of Culler--Vogtmann's outer space, which is closely related to lengths of conjugacy classes in $F_N$. A new difficulty arises however in the latter context: since the natural metric on outer space fails to be symmetric, outer space has in fact two horoboundaries (forward and backward), which appear to be rather different in nature (see \cite{Hor14}, where the forward horoboundary of outer space is completely described, and the geometry of the backward horoboundary is also investigated). We will actually only prove a central limit theorem for the Busemann cocycle on the backward horoboundary, but our arguments will require working with both boundaries.
\\
\\
\indent As a consequence of the work of Benoist--Quint, one can give a general condition under which the Busemann cocycle on the backward horoboundary $\partial_h^-X$ of a $G$-metric space $X$ (where $G$ is a countable group) satisfies a central limit theorem (a \emph{dual} version holds for the Busemann cocycle on the forward horoboundary $\partial_h^+X$ by reversing the roles of the forward and backward metrics). We denote by $\check{\mu}$ the \emph{reflected measure} on $G$, defined by letting $\check{\mu}(g):=\mu(g^{-1})$ for all $g\in G$. We denote by $(.|.)$ the natural extension of the Gromov product on $X$ to $\partial_h^-X\times\partial_h^+X$, defined by letting $$(x|y)_o:=-\frac{1}{2}\inf_{z\in X}(h^-_x(z)+h^+_y(z))$$ for all $x,y\in\partial_h^-X\times\partial_h^+X$, where $h_x^-$ and $h_y^+$ denote the functions on $X$ associated to $x$ and $y$. We denote by $d_X^{sym}$ the symmetrized metric on $X$, defined as the maximum of the forward and backward metrics.

\begin{theo}\label{intro-gen}
Let $(X,d_X)$ be a (possibly asymmetric) geodesic metric space, let $o\in X$, and let $G$ be a countable group acting by isometries on $X$. Let $\mu$ be a probability measure on $G$ with finite second moment with respect to $d_X^{sym}$. Assume that there exists a $G$-invariant measurable subset $Y^-\subseteq\partial_h^-X$, on which there exists a $\mu$-ergodic $\mu$-stationary probability measure $\nu$, and a $G$-invariant subset $Y^+\subseteq\partial_h^+X$, on which there exists a $\check{\mu}$-stationary probability measure $\nu^{\ast}$. Further assume that 
\begin{itemize}
\item \textbf{(H1)} there exists $\lambda\in\mathbb{R}$ such that $$\int_{G\times Y^+}\beta^+(g,y)d\check{\mu}(g)d\nu^{\ast}(y)=\lambda,$$ and 
\item \textbf{(H2)} there exists $\alpha>0$ and a sequence $(C_n)_{n\in\mathbb{N}}\in l^1(\mathbb{N})$ such that $$\nu^{\ast}(\{y\in Y^+|(x|y)_o\ge\alpha n\})\le C_n.$$
\end{itemize}
\noindent Then $\beta^{-}_{|Y^-}$ is centerable. Letting $$V_{\nu}:=\int_{G\times Y^-}(\beta_0^-(g,x)-\lambda)^2d\mu(g)d\nu(x),$$ and let $N_{\nu}$ be the centered Gaussian law on $\mathbb{R}$ with variance $V_{\nu}$. Then  $$\lim_{n\to +\infty}\int_{G} F \left(\frac{\beta^-(g,x)-n\lambda}{\sqrt{n}}\right)d\mu^{\ast n}(g) = \int_{\mathbb{R}}F(t)dN_{\nu}(t)$$ for $\nu$-a.e. $x\in Y^-$ and every compactly supported continuous function $F$ on $\mathbb{R}$.
\qed
\end{theo}

To derive Theorem \ref{intro-gen} from Benoist--Quint's work, one is essentially left solving the cohomological equation \eqref{coho} for the cocycle $\beta^-$: the solution is given as in \cite{BQ14} by 
\begin{equation}\label{sol}
\psi(x):=-2\int_{Y^+}(x|y)_{o}d\nu^{\ast}(y).
\end{equation}
Hypothesis \textbf{(H2)} ensures that $\psi$ is finite and bounded, and Hypothesis \textbf{(H1)} ensures that the cocycle $\beta_0^-$ defined as in \eqref{coho} has constant average $\lambda$. 
\\
\\
\indent The intuition behind Hypothesis \textbf{(H1)} is that if you pick an element $g\in G$ at random with respect to the probability measure $\mu$, and a random point $y$ in the forward horoboundary of $X$, then in average you will tend to move away from $y$ (along distance $\lambda$) when going from $o$ to $go$. This is a typical behaviour if $X$ is a hyperbolic space.

The intuition behind Hypothesis \textbf{(H2)} is that if you pick two points $x$ and $y$ in the backward and forward horoboundaries of $X$, then with high probability, geodesic rays from $o$ to $x$ (for the backward metric) and from $o$ to $y$ (for the forward metric) will rapidly diverge. This is again a typical behaviour if $X$ is a hyperbolic space. 
\\
\\
\indent In order to prove Theorem \ref{intro}, we will establish Hypotheses \textbf{(H1)} and \textbf{(H2)} for the Gromov product on the Teichmüller space $\mathcal{T}(S)$. The rough intuition is that though not hyperbolic, Teichmüller space is \emph{hyperbolic in average}, and typical rays in $\mathcal{T}(S)$ contain infinitely many subsegments with hyperbolic-like behaviour (see \cite{DDM14}). 

To make sense of this intuition, we will take advantage of the mapping class group action on the curve graph $\mathcal{C}(S)$ of the surface, which is known to be Gromov hyperbolic thanks to work by Masur--Minsky \cite{MM99}. Combining Benoist--Quint's arguments with recent work by Maher--Tiozzo \cite{MT15} extending results about random walks on hyperbolic spaces to non-proper settings, we will first establish Hypotheses \textbf{(H1)} and \textbf{(H2)} for the Gromov product on $\mathcal{C}(S)$. 

There is a well-behaved Lipschitz map from $\mathcal{T}(S)$ to $\mathcal{C}(S)$. In order to obtain the desired deviation estimates for the realization of the random walk on $\mathcal{T}(S)$, we will \emph{lift} to $\mathcal{T}(S)$ our estimates for the realization of the random walk on $\mathcal{C}(S)$. This is done by appealing to a contraction property of typical geodesics in $\mathcal{T}(S)$, following a strategy that was already used in \cite{DH15} for establishing spectral theorems for the random walks on $\text{Mod}(S)$ and $\text{Out}(F_N)$. Since the realizations of the random walk on $\mathcal{T}(S)$ and $\mathcal{C}(S)$ both escape the origin with positive speed, typical rays in $\mathcal{T}(S)$ must contain subsegments whose projections to $\mathcal{C}(S)$ make definite progress. Any subsegment $I$ that makes progress also satisfies the following contraction property: any other Teichmüller segment with the same projection to $\mathcal{C}(S)$ as $I$ passes uniformly close to $I$ in $\mathcal{T}(S)$. This will be the key observation for establishing Hypotheses \textbf{(H1)} and \textbf{(H2)} for the Gromov product on $\mathcal{T}(S)$. 
\\
\\
\indent We use a similar strategy for establishing Hypotheses \textbf{(H1)} and \textbf{(H2)} for Culler--Vogtmann's outer space. This time, we will take advantage of the action of $\text{Out}(F_N)$ on the so-called \emph{free factor graph}, which was proved to be Gromov hyperbolic by Bestvina--Feighn \cite{BF12}. However, new technical difficulties arise, mainly coming from the asymmetry of the metric on outer space, and the contraction property we establish in this context is slightly weaker than the one we use in the context of mapping class groups.

\paragraph*{Structure of the paper.} The paper is organized as follows. In Section \ref{sec-bus}, we establish a central limit theorem for Busemann cocycles, in the general case of metric spaces that may fail to be symmetric. In Section \ref{sec-hyp}, we build on the works of Benoist--Quint \cite{BQ14} and Maher--Tiozzo \cite{MT15} to establish quantitative deviation estimates for random walks on (possibly non-proper) hyperbolic spaces, under a second moment condition on the measure. The proof of the central limit theorem in the context of mapping class groups is carried in Section \ref{sec-mod}, and we deal with the $\text{Out}(F_N)$ case in Section \ref{sec-out-1}.

\paragraph*{Acknowledgments.} I would like to thank Yves Guivarc'h and Anders Karlsson for enlightening discussions and their interest in this project. I acknowledge support from ANR-11-BS01-013 and from the Lebesgue Center of Mathematics.

\section{A central limit theorem for Busemann cocycles}\label{sec-bus}

\subsection{Random walks on groups: general definitions and notations}

\paragraph{General notations.} Let $G$ be a group, and let $\mu$ be a probability measure on $G$. The \emph{(left) random walk} on $G$ with respect to the measure $\mu$ is the Markov chain on $G$ with initial distribution the Dirac measure at the identity element, and with transition probabilities $p(x,y):=\mu(yx^{-1})$. We warn the reader that we will always be considering left random walks on groups in the present paper, because it is more natural when applying a random product of diffeomorphisms to a curve, or a random product of outer automorphisms to a conjugacy class. However, many results concerning random walks on either mapping class groups or outer automorphism groups of free groups are stated for right random walks in the literature (which is more natural when considering the random walk as an actual walk at random on the Cayley graph of the group, or on any space it acts on). The product probability space $\Omega:=(G^{\mathbb{N}^{\ast}},\mu^{\otimes\mathbb{N}^{\ast}})$ is the space of increments of the random walk. The position of the random walk at time $n$ is given from its position $g_0=e$ at time $0$ by successive multiplications on the left of independent $\mu$-distributed increments $s_i$, i.e. $g_n=s_n\dots s_1$. The \emph{path space} $\mathcal{P}:=G^{\mathbb{N}}$ is equipped with the $\sigma$-algebra generated by the cylinders $\{\textbf{g}\in\mathcal{P}|g_i=g\}$ for all $i\in\mathbb{N}$ and all $g\in G$, and the probability measure $\mathbb{P}$ induced by the map 
\begin{displaymath}
\begin{array}{cccc}
\Omega &\to & \mathcal{P}\\
(s_1,s_2,\dots)&\mapsto &(g_0,g_1,g_2,\dots).
\end{array}
\end{displaymath}

\paragraph{Moment and drift.} Assume now that $G$ acts by isometries on a (possibly asymmetric) metric space $(X,d_X)$ (i.e. $d_X$ is assumed to satisfy the separation axiom and the triangle inequality, but it may fail to be symmetric). Let $o\in X$ be a basepoint. We say that $\mu$ has \emph{finite first moment} with respect to $d_X$ if $$\int_{G}d_X(go,o)d\mu(g)<+\infty.$$ It has \emph{finite second moment} with respect to $d_X$ if $$\int_{G}d_X(go,o)^2d\mu(g)<+\infty.$$ By Kingman's subadditive ergodic theorem \cite{Kin68}, if $\mu$ has finite first moment with respect to $d_X$, then for $\mathbb{P}$-a.e. sample path $(g_n)_{n\in\mathbb{N}}$ of the random walk on $(G,\mu)$, the limit $$\lim_{n\to +\infty}\frac{1}{n}d_X(g_no,o)$$ exists and is equal to $$\inf_{n\in\mathbb{N}}\frac{1}{n}\int_Gd_X(go,o)d\mu^{\ast n}(g),$$ where $\mu^{\ast n}$ denotes the $n^{\text{th}}$ convolution of $\mu$. This limit is called the \emph{drift} of the random walk on $(G,\mu)$ with respect to $d_X$. 

The \emph{reflected measure} $\check{\mu}$ is the probability measure on $G$ defined by letting $\check{\mu}(g):=\mu(g^{-1})$ for all $g\in G$. Notice that if $d_X$ is symmetric, and if $\mu$ has finite first (or second) moment with respect to $d_X$, then the same also holds for $\check{\mu}$. The symmetry of $d_X$, together with the fact that the $G$-action on $X$ is by isometries, also implies that the drifts of the random walks on $(G,\mu)$ and $(G,\check{\mu})$ with respect to $d_X$ are equal in this situation.

\paragraph{Stationary measures.} A probability measure $\nu$ on $X$ is said to be \emph{$\mu$-stationary} if $\nu=\mu\ast\nu$, where we recall that the convolution $\mu\ast\nu$ is the probability measure on $X$ given by $$\mu\ast\nu(S):=\int_{G}\nu(h^{-1}S)d\mu(h)$$ for all measurable subsets $S\subseteq X$. Any compact space admits a $\mu$-stationary probability measure, obtained as a weak limit of the Cesàro averages of the measures $\mu^{\ast n}\ast\delta_o$, where $\delta_o$ is the Dirac measure at $o$. 

\subsection{Horoboundaries and Busemann cocycles}\label{sec-horo}

\paragraph*{Horoboundaries.} Let $(X,d_X)$ be a (possibly asymmetric) geodesic metric space. We let $d_X^+:=d_X$, and $d^-_X$ be the (possibly asymmetric) metric on $X$ defined by letting $d^-_X(x,y):=d_X(y,x)$ for all $x,y\in X$. We also let $d^{sym}_X:=\max(d_X^+,d^-_X)$, which is a symmetric metric on $X$. Letting 
\begin{displaymath}
\begin{array}{cccc}
h_z^+:&X&\to&\mathbb{R}\\
&x&\mapsto &d_X^+(x,z)- d_X^+(o,z)
\end{array}
\end{displaymath}
\noindent for all $z\in X$, one obtains a continuous map
\begin{displaymath}
\begin{array}{cccc}
h^+ :& X&\to &\mathcal{C}(X)\\
& z&\mapsto& h_z^+
\end{array}
\end{displaymath}
\noindent where $\mathcal{C}(X)$ denotes the space of continuous real-valued functions on $X$, equipped with the topology of uniform convergence on compact sets of $(X,d_X^{sym})$. The closure $\text{cl}_h^+(X):=\overline{h^+(X)}$ in $\mathcal{C}(X)$ is compact. In the case where $(X,d_X^{sym})$ is a proper metric space, and that $d_X^+$ and $d_X^-$ determine the same topology on $X$, then the embedding of $X$ in $\overline{h^+(X)}$ is a homeomorphism onto its image \cite[Proposition 2.2]{Wal14}. In this situation, the \emph{horofunction boundary} $\partial_h^+X:=\overline{h^+(X)}\smallsetminus h^+(X)$ is compact. We will denote by $\partial_h^-X$ the horofunction boundary of $X$ for the metric $d^-_X$. 

\paragraph*{Extension of the Gromov product to the horoboundary.} For all $x,y\in X$, the \emph{Gromov product} of $x$ and $y$ with respect to $o$ is defined as $$(x|y)_{o}:=\frac{1}{2}(d_X(x,o)+d_X(o,y)-d_X(x,y)).$$ We extend it to $\text{cl}_h^-X\times\text{cl}_h^+X$ by letting $$(x|y)_o:=-\frac{1}{2}\inf_{z\in X}(h^-_x(z)+h^+_y(z))$$ for all $x\in\text{cl}_h^-X$ and all $y\in\text{cl}_h^+X$ (where we denote by $h^-_x$ and $h^+_y$ the functions on $X$ corresponding to $x$ and $y$). When $x,y\in X$, this indeed coincides with the Gromov product defined above: in this case, the infimum is achieved at any point $z\in X$ lying on a geodesic segment from $x$ to $y$. We note that $(x|y)_o$ may be infinite. We also note that this extension is not always continuous, see the example attributed to Walsh in \cite[Appendix]{Miy14}.

\paragraph*{Busemann cocycles.} Let now $G$ be a group acting by isometries on $X$. Then the $G$-action on $X$ extends continuously to an action by homeomorphisms on $\partial_h^+X$ by letting $$g.h^+_x(z):=h^+_x(g^{-1}.z)-h^+_x(g^{-1}.o)$$ for all $x\in\partial_hX$ and all $z\in X$. The \emph{Busemann cocycle} $\beta^+_X:G\times \text{cl}_h^+X\to\mathbb{R}$ is the continuous cocycle defined by $$\beta^+_X(g,x):=h^+_x(g^{-1}.o)$$ for all $(g,x)\in G\times \text{cl}_h^+X$ (recall that if $Y$ is a $G$-space, a map $\sigma:G\times Y\to\mathbb{R}$ is a \emph{cocycle} if $\sigma(gh,y)=\sigma(g,hy)+\sigma(h,y)$ for all $g,h\in G$ and all $y\in Y$). We similarly define a cocycle $\beta_X^-$ on $G\times\text{cl}_h^-X$. Notice that 
\begin{equation}\label{bus}
|\beta^+_X(g,x)|\le d_X^{sym}(go,o)
\end{equation}
for all $x\in \text{cl}_h^+ X$ and all $g\in G$.

\subsection{Deviation estimates for cocycles: review of Benoist--Quint's work}

Let $G$ be a countable group acting continuously on a compact metrizable space $X$. Given a continuous cocycle $\sigma:G\times X\to\mathbb{R}$, we let $\sigma_{sup}:G\to\mathbb{R}$ be the function defined by $$\sigma_{sup}(g)=\sup_{x\in X}|\sigma(g,x)|$$ for all $g\in G$.

\begin{prop}(Benoist--Quint \cite[Proposition 3.2]{BQ13})\label{bq}
Let $G$ be a discrete group, let $X$ be a compact metrizable $G$-space, let $\mu$ be a probability measure on $G$. Let $\sigma:G\times X\to\mathbb{R}$ be a continuous cocycle such that $\sigma_{sup}\in L^2(G,\mu)$. Assume that there exists $\lambda\in\mathbb{R}$ such that for all $\mu$-stationary probability measures $\nu$ on $X$, one has $$\int_{G\times X}\sigma(g,x)d\mu(g)d\nu(x)=\lambda.$$ Then for all $\epsilon>0$, there exists a sequence $(C_n)_{n\in\mathbb{N}}\in l^1(\mathbb{N})$ such that for all $n\in\mathbb{N}$ and all $x\in X$, one has $$\mu^{\ast n}(\{g\in G||\sigma(g,x)-n\lambda|\ge\epsilon n\})\le C_n.$$ 
\end{prop}

\subsection{Central limit theorem for Busemann cocycles}

We will specify Benoist--Quint's central limit theorem for centerable cocycles \cite[Theorem 3.4]{BQ13} to the specific case of the Busemann cocycle on the horoboundary of a metric space (Theorem \ref{tcl-gen} below). In particular, we give a criterion, coming from \cite{BQ14}, ensuring centerability of the Busemann cocycle (see below for a definition). 

We start by recalling some terminology. Let $G$ be a discrete group acting continuously on a metrizable $G$-space $X$, let $\mu$ be a probability measure on $G$, and let $\sigma:G\times X\to\mathbb{R}$ be a cocycle. Given $\lambda\in\mathbb{R}$, we say that $\sigma$ \emph{has constant drift $\lambda$ with respect to $\mu$} if for all $x\in X$, one has $$\int_{G}\sigma(g,x)d\mu(g)=\lambda.$$ We say that $\sigma$ is \emph{$\mu$-centerable} if there exist a measurable cocycle $\sigma_0:G\times X\to\mathbb{R}$ with constant drift with respect to $\mu$, and a bounded measurable function $\psi:X\to\mathbb{R}$, such that for all $g\in G$ and all $x\in X$, one has $$\sigma(g,x)=\sigma_0(g,x)+\psi(x)-\psi(gx).$$ The \emph{average} of $\sigma$ with respect to $\mu$ is defined as the drift of $\sigma_0$ with respect to $\mu$. The following proposition gives a criterion for ensuring $\mu$-centerability of the Busemann cocycle $\beta^-$ on the backward horoboundary of $X$ (by reversing the roles of $d_X$ and $d_X^-$, a dual criterion can also be given for the cocycle $\beta^+$).

\begin{prop}\label{centerable-gen}
Let $(X,d_X)$ be a (possibly asymmetric) geodesic metric space, let $o\in X$, and let $G$ be a countable group acting by isometries on $X$. Let $\mu$ be a probability measure on $G$. Assume that there exist $G$-invariant measurable subsets $Y^-\subseteq\partial_h^-X$ and $Y^+\subseteq\partial_h^+X$, and a $\check{\mu}$-stationary probability measure $\nu^{\ast}$ on $Y^+$. Further assume that 
\begin{itemize}
\item \textbf{(H1)} there exists $\lambda\in\mathbb{R}$ such that $$\int_{G\times Y^+}\beta^+(g,y)d\check{\mu}(g)d\nu^{\ast}(y)=\lambda,$$ and 
\item \textbf{(H2)} there exists $\alpha>0$ and a sequence $(C_n)_{n\in\mathbb{N}}\in l^1(\mathbb{N})$ such that for all $x\in Y^-$, one has $$\nu^{\ast}(\{y\in Y^+|(x|y)_o\ge\alpha n\})\le C_n.$$
\end{itemize}
\noindent Then $\beta^-_{|Y^-}$ is $\mu$-centerable with average $\lambda$.
\end{prop}

\begin{proof}
The argument follows the proofs of \cite[Propositions 4.2 and 4.6]{BQ14}. For all $x\in Y^-$, we let $$\psi(x):=-2\int_{Y^+}(x|y)_od\nu^{\ast}(y).$$ Hypothesis \textbf{(H2)} implies that $\psi$ is finite and bounded: indeed, one bounds the integral by cutting $Y^+$ into subsets of the form $\{y\in Y^+|\alpha n\le (x|y)_o\le\alpha(n+1)\}$ with $n$ varying over $\mathbb{N}$, and then using summability of the sequence $(C_n)_{n\in\mathbb{N}}$. We then note that 
\begin{equation}\label{bq1}
\beta^-(g,x)=\beta^+(g^{-1},y)+2(gx|y)_o-2(x|g^{-1}y)_o
\end{equation}
for all $g\in G$, all $x\in Y^-$ and all $y\in Y^+$ (this relation is a consequence of the definitions; it was already noticed by Benoist--Quint \cite[Lemma 1.2]{BQ14} in the case of a symmetric metric space). Since $\nu^{\ast}$ is $\check{\mu}$-stationary, we have $$-2\int_{G\times Y^+}(x|g^{-1} y)_od\mu(g)d\nu^{\ast}(y)=-2\int_{Y^+}(x|y)_od\nu^{\ast}(y)=\psi(x).$$ Using in addition Hypothesis \textbf{(H1)}, we obtain, by integrating \eqref{bq1} on $G\times Y^+$ with respect to $d\mu(g)d\nu^{\ast}(y)$, that $$\int_{G}\beta^-(g,x)d\mu(g)=\lambda-\int_{G}\psi(g.x)d\mu(g)+\psi(x)$$ for all $x\in Y^-$. The cocycle $\beta^-_0$ defined by letting $$\beta_0^-(g,x):=\beta^-(g,x)+\psi(g.x)-\psi(x)$$ for all $g\in G$ and all $x\in Y^-$ has constant drift $\lambda$ with respect to $\mu$. Hence $\beta^-_{|Y^-}$ is $\mu$-centerable with average $\lambda$. 
\end{proof}

Once centerability of $\beta^-_{|Y^-}$ is established, the proof of the following central limit theorem is the same as the proofs of \cite[Theorem 3.4]{BQ13} or \cite[Theorem 4.7]{BQ14}, which rely on a central limit theorem for martingales due to Brown \cite{Bro71}. The cocycle $\beta_0^-$ appearing in the definition of $V_{\nu}$ is a cocycle provided by $\mu$-centerability of $\beta^-_{|Y^-}$. As noticed in \cite[Remark 3.3]{BQ13}, the value of $V_{\nu}$ does not depend on the choice of $\beta_0^-$.

\begin{theo}\label{tcl-gen}
Let $(X,d_X)$ be a (possibly asymmetric) geodesic metric space, let $o\in X$, and let $G$ be a countable group acting by isometries on $X$. Let $\mu$ be a probability measure on $G$ with finite second moment with respect to $d_X^{sym}$. Assume that there exists a $G$-invariant measurable subset $Y^-\subseteq\partial_h^-X$, on which there exists a $\mu$-ergodic $\mu$-stationary probability measure $\nu$, and a $G$-invariant subset $Y^+\subseteq\partial_h^+X$, on which there exists a $\check{\mu}$-stationary probability measure $\nu^{\ast}$. Further assume that 
\begin{itemize}
\item \textbf{(H1)} there exists $\lambda\in\mathbb{R}$ such that $$\int_{G\times Y^+}\beta^+(g,y)d\check{\mu}(g)d\nu^{\ast}(y)=\lambda,$$ and 
\item \textbf{(H2)} there exists $\alpha>0$ and a sequence $(C_n)_{n\in\mathbb{N}}\in l^1(\mathbb{N})$ such that $$\nu^{\ast}(\{y\in Y^+|(x|y)_o\ge\alpha n\})\le C_n.$$
\end{itemize}
\noindent Let $$V_{\nu}:=\int_{G\times Y^-}(\beta_0^-(g,x)-\lambda)^2d\mu(g)d\nu(x),$$ and let $N_{\nu}$ be the centered Gaussian law on $\mathbb{R}$ with variance $V_{\nu}$. Then  $$\lim_{n\to +\infty}\int_{G} F \left(\frac{\beta^-(g,x)-n\lambda}{\sqrt{n}}\right)d\mu^{\ast n}(g) = \int_{\mathbb{R}}F(t)dN_{\nu}(t)$$ for $\nu$-a.e. $x\in Y^-$ and every compactly supported continuous function $F$ on $\mathbb{R}$.
\qed
\end{theo}

\begin{rk}
If we assumed in addition that $Y^-=\partial_h^-X$, and that $\beta^-$ has unique covariance in the sense of \cite[Section 3.3]{BQ13} (which happens in particular if $\partial_h^-X$ carries a unique $\mu$-stationary probability measure), then Theorem \ref{tcl-gen} would be a specification of Benoist--Quint's central limit theorem for centerable cocycles on compact spaces \cite[Theorem 3.4]{BQ13}, in which case the convergence would be uniform in $x\in\partial_h^-X$ (notice that square-integrability of $\beta^-$ follows from the second moment assumption on $\mu$ together with Equation \eqref{bus}). Nevertheless, Benoist--Quint's proof, as it appears in \cite[Theorem 4.7]{BQ14} applies as such (with $Y^-$ in place of Benoist--Quint's space $X$) to prove Theorem \ref{tcl-gen}.     
\end{rk}

\section{Deviation results for random walks on hyperbolic spaces}\label{sec-hyp}

\subsection{Review on Gromov hyperbolic spaces}\label{sec-11}

We briefly review basic facts about Gromov hyperbolic spaces, and refer the reader to \cite{GdlH90} for details. A symmetric metric space $X$ is \emph{Gromov hyperbolic} if there exists $\delta\ge 0$ such that for all $x,y,z,o\in X$, one has $$(x|y)_{o}\ge\min ((x|z)_{o},(y|z)_{o})-\delta.$$ The smallest such $\delta$ is then called the \emph{hyperbolicity constant} of $X$. 

From now on, we let $X$ be a geodesic Gromov hyperbolic metric space. A sequence $(x_n)_{n\in\mathbb{N}}\in X^{\mathbb{N}}$ \emph{converges to infinity} if $(x_n|x_m)_{o}$ goes to $+\infty$ as $n$ and $m$ go to $+\infty$. Two sequences $(x_n)_{n\in\mathbb{N}},(y_n)_{n\in\mathbb{N}}\in X^{\mathbb{N}}$ that converge to infinity are said to be \emph{equivalent} if $(x_n|y_m)_{o}$ goes to $+\infty$ as $n$ and $m$ go to $+\infty$. The \emph{Gromov boundary} $\partial_{\infty} X$ of $X$ is the set of equivalence classes of sequences that converge to infinity. The Gromov product on $X$ extends to $X\cup\partial_{\infty}X$ by letting $$\langle a|b\rangle_{o}:=\inf\liminf_{i,j\to +\infty}\langle x_i|y_j\rangle_{o}$$ for all $a,b\in X\cup\partial_{\infty}X$, the infimum being taken over all sequences $(x_i)_{i\in\mathbb{N}}\in X^{\mathbb{N}}$ converging to $a$ and all sequences $(y_j)_{j\in\mathbb{N}}\in X^{\mathbb{N}}$ converging to $b$. 

In the case of Gromov hyperbolic geodesic metric spaces, we have defined two extensions of the Gromov product: the extension to the Gromov boundary (denoted with brackets), and the extension to $\text{cl}_hX$ from the previous section (denoted with parentheses). Notice that since the metric on $X$ is assumed to be symmetric in this section, we will just write $\text{cl}_hX$ without mentioning any superscript. The two extensions of the Gromov product are related in the following way. Following Maher--Tiozzo \cite[Section 3.2]{MT15}, we define $\text{cl}_h^{\infty}X$ as the subspace of $\text{cl}_hX$ made of those horofunctions $h$ such that $\inf_{x\in X}h(x)=-\infty$. Then there is a projection map $\pi:X\cup\text{cl}_h^{\infty}X\to X\cup\partial_{\infty}X$ which restricts to the identity on $X$, and such that for all $h\in\text{cl}_h^{\infty}X$, all sequences $(x_n)_{n\in\mathbb{N}}\in X^{\mathbb{N}}$ converging to $h$ for the topology on $\text{cl}_hX$, also converge to $\pi(h)$ for the topology on $X\cup\partial_{\infty}X$. All results below are adaptations of observations made by Benoist--Quint \cite[Section 2]{BQ14} to the case of a possibly non-proper Gromov hyperbolic geodesic metric space $X$.

\begin{lemma}\label{lemma-ccc-1}
There exists $C>0$, only depending on the hyperbolicity constant of $X$, such that $$\langle\pi(x)|\pi(y)\rangle_{o}-C\le (x|y)_{o}\le \langle\pi(x)|\pi(y)\rangle_{o}+C$$ for all $x,y\in X\cup \text{cl}_h^{\infty}X$ satisfying $\pi(x)\neq\pi(y)$.
\end{lemma}

Given $C,K>0$, a \emph{$(C,K)$-quasigeodesic} in $X$ is a map $\gamma:\mathbb{R}\to X$ such that $$\frac{1}{C}|s-t|-K\le d_X(\gamma(s),\gamma(t))\le C|s-t|+K$$ for all $s,t\in\mathbb{R}$.

\begin{proof}[Proof of Lemma \ref{lemma-ccc-1}]
The key observation is that there exists a constant $K>0$, only depending on the hyperbolicity constant of $X$, such that for all $x,y\in X\cup\text{cl}_h^{\infty}X$ satisfying $\pi(x)\neq\pi(y)$, there exists a $(1,K)$-quasigeodesic $\gamma:I\to X$ (where $I\subseteq\mathbb{R}$ is either an interval, a half-line, or $I=\mathbb{R}$) such that $\gamma(t)$ converges to $\pi(x)$ (resp. $\pi(y)$) as $t$ goes to $-\infty$ (resp. $+\infty$). One then notices that up to a bounded additive error, the infimum in the formula defining $(x|y)_o$ can be taken over all points $z$ lying on the image of $\gamma$. 
\end{proof}

\begin{lemma}\label{lemma-hyp-0}
For all $x,y\in\text{cl}_h^{\infty}X$ such that $\pi(x)\neq\pi(y)$, there exists $C_{x,y}>0$ such that $$\max(h_x(m),h_y(m))\ge d_X(o,m)-C_{x,y}$$ for all $m\in X$. 
\end{lemma}

\begin{proof}
One has to choose $C_{x,y}$ to be sufficiently large compared to twice the distance from $o$ to a quasigeodesic line joining $x$ to $y$. Details of the proof are an exercise in hyperbolic metric spaces, and left to the reader.
\end{proof}

\noindent Let now $G$ be a group acting by isometries on $X$. We let $$\kappa_X(g):=d_X(go,o)$$ for all $g\in G$. As a consequence of Equation \eqref{bus} from Section \ref{sec-horo} and Lemma \ref{lemma-hyp-0} applied to $m=g^{-1}o$, we obtain the following fact.

\begin{cor}\label{lemma-hyp}
For all $x,y\in\text{cl}_h^{\infty}X$ such that $\pi(x)\neq\pi(y)$, there exists $C_{x,y}>0$ such that $$\kappa_X(g)-C_{x,y}\le\max(\beta_X(g,x),\beta_X(g,y))\le\kappa_X(g)$$ for all $g\in G$.
\qed
\end{cor}

\begin{lemma}\label{lemma-hyp-2}
There exists $C>0$, only depending on the hyperbolicity constant of $X$, such that for all $g\in G$ and all $x\in\text{cl}_h^{\infty}X$, one has $$\left|(go|gx)_o-\frac{1}{2}(\kappa_X(g)+\beta_X(g,x))\right|\le C$$ and $$\left|(go|x)_o-\frac{1}{2}(\kappa_X(g)-\beta_X(g^{-1},x))\right|\le C.$$
\end{lemma}

\begin{proof}
This follows from the definitions and the hyperbolicity of $X$ by noticing that for all $m\in X$ and all $x\in\text{cl}_h^{\infty}X$, the infimum in the formula defining $(m|x)_o$ can be taken, up to a bounded additive error, over the points $z$ lying on a $(1,K)$-quasigeodesic ray from $m$ to $\pi(x)$.
\end{proof}

\subsection{A large deviation principle for the Busemann cocycle}

Building on work by Benoist--Quint \cite{BQ14} and Maher--Tiozzo \cite{MT15}, we will establish quantitative deviation results for random walks on groups acting on (possibly non-proper) hyperbolic spaces, under a second moment assumption on the measure. Deviation estimates were also obtained by Mathieu--Sisto \cite{MS14} under an exponential moment assumption on the measure. Throughout the section, we let $X$ be a separable Gromov hyperbolic geodesic metric space, and $G$ be a countable group acting by isometries on $X$. We fix a basepoint $o\in X$. A subgroup $H\subseteq G$ is \emph{nonelementary} if it contains two loxodromic isometries of $X$ with disjoint fixed point sets in $\partial_{\infty}X$. A probability measure $\mu$ on $G$ is \emph{nonelementary} if the subsemigroup of $G$ generated by the support of $\mu$ is a nonelementary subgroup of $G$. In this case, the reflected measure $\check{\mu}$ is also nonelementary. We start by recalling the following result of Maher--Tiozzo.

\begin{prop}(Maher--Tiozzo \cite[Theorem 1.1]{MT15})\label{MT}
Let $G$ be a countable group acting by isometries on a separable Gromov hyperbolic geodesic metric space $(X,d_X)$, let $o\in X$, and let $\mu$ be a nonelementary probability measure on $G$. Then for $\mathbb{P}$-a.e. every sample path $\mathbf{g}:=(g_n)_{n\in\mathbb{N}}$ of the random walk on $(G,\mu)$, the sequence $(g_n^{-1}.o)_{n\in\mathbb{N}}$ converges to a point $\text{bnd}(\mathbf{g})\in\partial_{\infty}X$. The hitting measure $\nu^{\ast}$ on $\partial_{\infty}X$ defined by letting $$\nu^{\ast}(S)=\mathbb{P}[\text{bnd}(\mathbf{g})\in S]$$ for all measurable subsets $S\subseteq\partial_{\infty}X$, is nonatomic, and it is the unique $\check{\mu}$-stationary probability measure on $\partial_{\infty}X$. 
\end{prop}

We will first prove a deviation principle for the Busemann cocycle $\beta_X:G\times\text{cl}_hX\to\mathbb{R}$, under a second moment assumption on $\mu$ (Proposition \ref{ccc} below). We recall that $\kappa_X(g):=d_X(go,o)$ for all $g\in G$. The following lemma is an extension of \cite[Proposition 3.2]{BQ14} to the case where $X$ is no longer assumed to be proper.

\begin{lemma}(Benoist--Quint \cite[Proposition 3.2]{BQ14})\label{lemma-ccc}
Let $G$ be a countable group acting by isometries on a separable Gromov hyperbolic metric space $X$, and let $\mu$ be a nonelementary probability measure on $G$. Then for all $\epsilon>0$, there exists $T>0$ such that for all $x\in\text{cl}_h^{\infty}X$, one has $$\mathbb{P}\left[\sup_{n\in\mathbb{N}}|\kappa_X(g_n)-\beta_X(g_n,x)|\le T\right]\le 1-\epsilon.$$ 
\end{lemma}

\begin{proof}
The proof is similar to the proof of \cite[Proposition 3.2]{BQ14}, by using the convergence statement recalled in Proposition \ref{MT} and the adaptation of Benoist--Quint's estimates given in Section \ref{sec-11}. The rough idea is to show that if $x\in\text{cl}_h^{\infty}X$, then $\mathbb{P}$-a.s., one has $$\sup_{n\in\mathbb{N}}|\kappa_X(g_n)-\beta_X(g_n,x)|<+\infty,$$ from which the lemma follows. This fact is shown by noticing that $\mathbb{P}$-a.s. the sequence $(g_n^{-1}.o)_{n\in\mathbb{N}}$ converges to a boundary point $\text{bnd}(\mathbf{g})$ distinct from $x$ (because the hitting measure is nonatomic), and denoting by $z$ a coarse center for the triangle made of $o, x$ and $\text{bnd}(\mathbf{g})$, then $|\kappa_X(g_n)-\beta_X(g_n,x)|$ is equal, up to a bounded additive error,  to $2d_X(o,z)$, for all sufficiently large $n\in\mathbb{N}$.  
\end{proof}

\begin{cor}\label{H1-hyp}
Let $G$ be a countable group acting by isometries on a separable Gromov hyperbolic metric space $X$, and let $\mu$ be a nonelementary probability measure on $G$ with finite first moment with respect to $d_X$. Let $\lambda_X$ be the drift of the random walk on $(G,\mu)$ with respect to $d_X$. Then $$\int_{G\times\text{cl}_hX}\beta_X(g,x)d\mu(g)d\theta(x)=\lambda_X$$ for all $\mu$-stationary probability measures $\theta$ on $\text{cl}_hX$.
\end{cor}

\begin{proof}
Lemma \ref{lemma-ccc} implies that for all $x\in\text{cl}_h^{\infty}X$ and $\mathbb{P}$-a.e. sample path $(g_n)_{n\in\mathbb{N}}$ of the random walk on $(G,\mu)$, one has $$\lim_{n\to +\infty}\frac{1}{n}\beta_X(g_n,x)=\lambda_X.$$ By \cite[Proposition 4.4]{MT15}, all $\mu$-stationary probability measures on $\text{cl}_hX$ are supported on $\text{cl}_h^{\infty}X$. As $\beta_X$ is a cocycle, Corollary \ref{H1-hyp} follows by applying Birkhoff's ergodic theorem.
\end{proof}

The following proposition can be viewed both as an extension of \cite[Proposition 4.1]{BQ14} to the case of a random walk on a group acting by isometries on a (not necessarily proper) hyperbolic space, and as an extension of \cite[Theorem 1.2]{MT15} to the case of a measure with finite second moment with respect to $d_X$. 

\begin{prop}(Benoist--Quint \cite[Proposition 4.1]{BQ14})\label{ccc}
Let $G$ be a countable group acting by isometries on a separable Gromov hyperbolic geodesic metric space $X$, and let $\mu$ be a nonelementary probability measure on $G$ with finite second moment with respect to $d_X$. Let $\lambda_X$ be the drift of the random walk on $(G,\mu)$ with respect to $d_X$. Then for all $\epsilon>0$, there exists a sequence $(C_n)_{n\in\mathbb{N}}\in l^1(\mathbb{N})$ such that for all $x\in\text{cl}_hX$ and all $n\in\mathbb{N}$, one has $$\mu^{\ast n}(\{g\in G||\beta_X(g,x)-n\lambda_X|\ge\epsilon n\})\le C_n,$$ and $$\mu^{\ast n}(\{g\in G||\kappa_X(g)-n\lambda_X|\ge\epsilon n\})\le C_n.$$
\end{prop}

\begin{proof}
The deviation principle for the Busemann cocycle $\beta_X$ follows from Corollary \ref{H1-hyp} and Benoist--Quint's large deviation principle for cocycles (Proposition \ref{bq}). Notice that metrizability of $\text{cl}_hX$ was established in \cite[Proposition 3.1]{MT15}. The deviation principle for the function $\kappa_X$ then follows from the deviation principle for the Busemann cocycle by using Corollary \ref{lemma-hyp}.
\end{proof}

\subsection{Sublinear tracking}\label{sec-tracking}

Using the fact that $\lambda_X>0$ by \cite[Theorem 1.2]{MT15}, and arguing as in \cite[Lemma 4.5]{BQ14}, one can deduce the following estimate from Proposition \ref{ccc}.

\begin{lemma}(Benoist--Quint \cite[Lemma 4.5]{BQ14})\label{lemma-shadow}
Let $G$ be a countable group acting by isometries on a separable Gromov hyperbolic geodesic metric space $X$, and let $\mu$ be a nonelementary probability measure on $G$, with finite second moment with respect to $d_X$. Let $\nu$ be the unique $\mu$-stationary probability measure on $\partial_{\infty}X$. Then for all $\alpha>0$, there exists a sequence $(C_n)_{n\in\mathbb{N}}\in l^1(\mathbb{N})$ such that for all $n\in\mathbb{N}$ and all $x,y\in X\cup\partial_{\infty}X$, one has $$\mu^{\ast n}(\{g\in G|\langle go|y\rangle_{o}\ge\alpha n\})\le C_n$$ and $$\mu^{\ast n}(\{g\in G|\langle gx|y\rangle_{o}\ge\alpha n\}\le C_n,$$ and hence $$\nu(\{x\in \partial_{\infty}X|\langle x|y\rangle_{o}\ge\alpha n\})\le C_n.$$
\end{lemma}

\begin{rk}
Corollary \ref{H1-hyp} guarantees that Hypothesis \textbf{(H1)} from Theorem \ref{tcl-gen} is satisfied, and Lemma \ref{lemma-shadow} implies that Hypothesis \textbf{(H2)} is also satisfies. One can deduce that the Busemann cocycle $\beta_X$ and the function $\kappa_X$ both satisfy a central limit theorem, extending Benoist--Quint's central limit theorem for hyperbolic groups \cite{BQ14} to non-proper settings.
\end{rk}

We recall from Proposition \ref{MT} that for $\mathbb{P}$-a.e. sample path $\mathbf{g}=(g_n)_{n\in\mathbb{N}}$ of the random walk on $(G,\mu)$, the sequence $(g_n^{-1}.o)_{n\in\mathbb{N}}$ converges to a boundary point $\text{bnd}(\mathbf{g})\in\partial_{\infty}X$. We will denote by $\tau_{\mathbf{g},X}$ a $(1,K)$-quasigeodesic ray from $o$ to $\text{bnd}(\mathbf{g})$ (where $K$ only depends on the hyperbolicity constant of $X$). We now obtain the following quantitative version of sublinear tracking under a second moment assumption on the measure $\mu$.

\begin{prop}\label{deviation}
Let $G$ be a countable group acting by isometries on a separable Gromov hyperbolic geodesic metric space $X$, and let $\mu$ be a nonelementary probability measure on $G$ with finite second moment with respect to $d_X$.\\ 
Then for all $\epsilon>0$, there exists a sequence $(C_n)_{n\in\mathbb{N}}\in l^1(\mathbb{N})$ such that $$\mathbb{P}\left[d_X(g_n^{-1}.o,\tau_{\mathbf{g},X}(\mathbb{R}_+))\ge\epsilon n\right]\le C_n.$$
\end{prop}

\begin{proof}
The proof follows the argument from the proof of \cite[Proposition 5.7]{MT15}, using the quantitative estimate from Lemma \ref{lemma-shadow}; it goes as follows. There exists $K>0$, only depending on the hyperbolicity constant of $X$, such that $$|d_X(g_n^{-1}.o,\tau_{\mathbf{g},X}(\mathbb{R}_+))- \langle o|\text{bnd}(\mathbf{g})\rangle_{g_n^{-1}.o}|\le K,$$ and hence $$|d_X(g_n^{-1}.o,\tau_{\mathbf{g},X}(\mathbb{R}_+))- \langle g_n.o|g_n.\text{bnd}(\mathbf{g})\rangle_{o}|\le K.$$ Therefore, the condition $$d_X(g_n^{-1}.o,\tau_{\mathbf{g},X}(\mathbb{R}_+))\ge\epsilon n$$ implies that $$\langle g_n.o|g_n.\text{bnd}(\mathbf{g})\rangle_o\ge\epsilon n-K.$$ In addition, the boundary point $g_n.\text{bnd}(\mathbf{g})$ is independent from $g_n$, and its distribution is given by the hitting measure $\nu^{\ast}$. Conditioning over the value of $g_n.\text{bnd}(\mathbf{g})$, we get $$\mathbb{P}\left[d_X(g_n^{-1}.o,\tau_{\mathbf{g},X}(\mathbb{R}))\ge\epsilon n\right]\le\int_{\partial_{\infty}X}\mu^{\ast n}(\{g\in G|\langle go|y\rangle_o\ge\epsilon n-K\})d\nu^{\ast}(y).$$ Proposition \ref{deviation} therefore follows from Lemma \ref{lemma-shadow}. 
\end{proof}

Given $\kappa>0$, two quasigeodesic segments $\gamma:[a,b]\to X$ and $\gamma':[a',b']\to X$ are said to \emph{fellow travel up to distance $\kappa$} if there exists an increasing homeomorphism $\theta:[a,b]\to [a',b']$ such that $d_X(\gamma(t),\gamma'\circ\theta(t))\le\kappa$ for all $t\in [a,b]$. 

\begin{prop}\label{crossing}
Let $G$ be a countable group acting by isometries on a separable Gromov hyperbolic geodesic metric space $X$, and let $\mu$ be a nonelementary probability measure on $G$ with finite second moment with respect to $d_X$. Let $\lambda_X$ be the drift of the random walk on $(G,\mu)$ with respect to $d_X$. Then for all $K>0$, there exists a constant $\kappa>0$, only depending on the hyperbolicity constant of $X$, such that for all $0<\beta_1<\beta_2$, there exists a sequence $(C_n)_{n\in\mathbb{N}}\in l^1(\mathbb{N})$, such that for all $x\in X\cup\partial_{\infty} X$, with probability at least $1-C_n$, any $(K,K)$-quasigeodesic from $x$ to $\text{bnd}(\mathbf{g})$ contains a subsegment that fellow travels $\tau_{\mathbf{g},X}([\beta_1 n,\beta_2n])$ up to distance $\kappa$.
\end{prop}

\begin{proof}
This follows from Lemma \ref{lemma-shadow}, giving the existence of a sequence $(C_n)_{n\in\mathbb{N}}\in l^1(\mathbb{N})$ such that for all $x\in X\cup\partial_{\infty}X$, one has $$\mathbb{P}\left[\langle x|\text{bnd}(\mathbf{g})\rangle_o\le\frac{\beta_1 n}{2}\right]\ge 1-C_n,$$ together with hyperbolicity of $X$.
\end{proof}

\section{Central limit theorem on mapping class groups}\label{sec-mod}

The goal of this section is to establish a central limit theorem in the context of mapping class groups of surfaces (Theorem \ref{intro}).

\subsection{Background on mapping class groups}

Let $S$ be a closed, connected, oriented, hyperbolic surface. The \emph{mapping class group} $\text{Mod}(S)$ is defined as the group of all isotopy classes of orientation-preserving diffeomorphisms of $S$. We start by reviewing classical material on mapping class groups. 

\paragraph*{Teichmüller space and two of its metrics.} 

The \emph{Teichmüller space} $\mathcal{T}(S)$ is the space of isotopy classes of conformal structures on $S$. Up to isotopy, there is a unique hyperbolic metric on $S$ in a given conformal class, so $\mathcal{T}(S)$ can alternatively be defined as the space of isotopy classes of hyperbolic metrics on $S$. We review the definition of two metrics on $\mathcal{T}(S)$.

The \emph{Teichmüller metric} is defined by letting $$d_{\mathcal{T}}(x,y):=\frac{1}{2}\inf_f\log K(f)$$ for all $x,y\in\mathcal{T}(S)$, where the infimum is taken over the collection of all quasiconformal maps $f$ from $(S,x)$ to $(S,y)$, and $K(f)$ denotes the quasiconformal dilatation of $f$. The Teichmüller metric is uniquely geodesic: any two points in $\mathcal{T}(S)$ are joined by a unique geodesic segment. This metric can alternatively be described in terms of ratios of extremal lengths of curves, as follows. A simple closed curve $c$ on $S$ is \emph{essential} if it does not bound a disk on $S$. Given an essential simple closed curve $c$ on $S$ and $x\in\mathcal{T}(S)$, the \emph{extremal length} of $c$ in the conformal structure $x$ is $$\text{Ext}_{x}(c)=\sup_{\rho}\frac{l_{\rho}(c)^2}{\text{Area}(\rho)},$$ where the supremum is taken over all metrics $\rho$ in the conformal class $x$, where $l_{\rho}(c)$ denotes the infimal $\rho$-length of a curve isotopic to $c$, and $\text{Area}(\rho)$ is the area of $S$ equipped with the metric $\rho$. We denote by $\mathcal{S}$ the collection of all isotopy classes of essential simple closed curves on $S$. Kerckhoff proved in \cite[Theorem 4]{Ker80} that $$d_{\mathcal{T}}(x,y)=\frac{1}{2}\log\sup_{c\in\mathcal{S}}\frac{\text{Ext}_y(c)}{\text{Ext}_x(c)}$$ for all $x,y\in\mathcal{T}(S)$. 

\emph{Thurston's asymmetric metric} is defined by letting $$d_{Th}(x,y):=\inf_f\log \text{Lip}(f)$$ for all $x,y\in\mathcal{T}(S)$, where the infimum is taken over the collection of all Lipschitz maps $f$ from $(S,x)$ to $(S,y)$, and $\text{Lip}(f)$ denotes the Lipschitz constant of $f$. This metric can also be described using hyperbolic lengths of curves: Thurston established that $$d_{Th}(x,y)=\log\sup_{c\in\mathcal{S}}\frac{l_y(c)}{l_x(c)}$$ for all $x,y\in\mathcal{T}(S)$, where $l_x$ denotes the length measured in the unique hyperbolic metric in the conformal class $x$. The next proposition, due to Lenzhen--Rafi--Tao, states that, up to a bounded additive error, one can actually take the supremum over a finite collection of curves in the above formula. Given $x\in\mathcal{T}(S)$, a \emph{marking} $\mu_x$ on $(S,x)$ is a collection of curves of $S$, made of both a finite set $\mathcal{P}$ of pairwise disjoint curves on $S$ that cut $S$ into a finite collection of pairs of pants, and a set of transverse curves $\mathcal{Q}$ that satisfy the following property: each curve $\alpha\in\mathcal{P}$ intersects exactly one curve $\beta\in\mathcal{Q}$, and $\beta$ intersects $\alpha$ minimally, and does not intersect any other curve in $\mathcal{P}$. The marking $\mu_x$ is a \emph{short marking} if $\mathcal{P}$ is constructed by first picking a shortest curve on $S$, then a second shortest curve, and so on, and curves in $\mathcal{Q}$ are then chosen to be as short as possible.

\begin{prop}(Lenzhen--Rafi--Tao \cite[Theorem E]{LRT12})\label{metric-lengths}
There exists $K>0$ (only depending on the topological type of $S$) such that for all $x,y\in\mathcal{T}(S)$, one has $$\left|d_{Th}(x,y)-\log\max_{c\in\mu_x}\frac{l_y(c)}{l_x(c)}\right|\le K,$$ where the maximum is taken over the collection of all curves in a short marking $\mu_x$ on $(S,x)$. 
\end{prop}

Given $\epsilon>0$, the \emph{$\epsilon$-thick part} $\mathcal{T}(S)^{\epsilon}$ is the subspace of $\mathcal{T}(S)$ made of those hyperbolic metrics for which no essential simple closed curve on $S$ has length smaller than $\epsilon$. The two metrics defined above are comparable in restriction to the thick part of $\mathcal{T}(S)$.

\begin{prop}(Choi--Rafi \cite[Theorem B]{CR07})\label{metrics}
For all $\epsilon>0$, there exists $C=C(\epsilon)>0$ such that for all $x,y\in\mathcal{T}(S)^{\epsilon}$, one has $|d_{\mathcal{T}}(x,y)-d_{Th}(x,y)|\le C$.
\end{prop}

\paragraph*{Thurston's and Gardiner--Masur's boundaries.} Thurston defined a compactification of $\mathcal{T}(S)$, as the closure of the image of the embedding 
\begin{displaymath}
\begin{array}{cccc}
\mathcal{T}(S)&\to &\mathbb{PR}^{\mathcal{S}}\\
x&\mapsto & \mathbb{R}^{\ast}(l_x(c))_{c\in\mathcal{S}}
\end{array}
\end{displaymath}
\noindent and he identified the boundary $\overline{\mathcal{T}(S)}\smallsetminus\mathcal{T}(S)$ with the space $\mathcal{PMF}$ of projective Whitehead equivalence classes of measured foliations on $S$, see \cite{FLP79}. We denote by $\mathcal{MF}$ the space of (unprojectivized) Whitehead equivalence classes of measured foliations on $S$. The length pairing between curves in $\mathcal{S}$ and points in $\mathcal{T}(S)$ extends to continuous intersection pairings (denoted by $i$) from $\mathcal{MF}\times\mathcal{T}(S)$ to $\mathbb{R}$ and from $\mathcal{MF}\times\mathcal{MF}$ to $\mathbb{R}$.

A measured foliation $F$ on $S$ is \emph{arational} if no leaf of $F$ contains a simple closed curve on $S$. It is \emph{uniquely ergodic} if in addition, every measured foliation $F'\in\mathcal{MF}$ with the same topological support as $F$ is homothetic to $F$. We will denote by $\mathcal{UE}\subseteq\mathcal{PMF}$ the space of uniquely ergodic arational foliations, and let $\mathcal{PMF}_0:=\mathcal{UE}\cup\mathcal{S}$. Given any two transverse measured foliations $x,y\in\mathcal{PMF}_0$, there exists a Teichmüller geodesic $\gamma:\mathbb{R}\to\mathcal{T}(S)$ such that $\gamma(t)$ converges to $x$ (resp. to $y$) as $t$ goes to $-\infty$ (resp. $+\infty$).

Gardiner and Masur have constructed \cite{GM91} another compactification $\text{cl}_{GM}\mathcal{T}(S)$ of $\mathcal{T}(S)$, using extremal lengths instead of hyperbolic lengths, by taking the closure of the image of the embedding
\begin{displaymath}
\begin{array}{cccc}
\mathcal{T}(S)&\to &\mathbb{PR}^{\mathcal{S}}\\
x&\mapsto & \mathbb{R}^{\ast}(\text{Ext}_x(c)^{\frac{1}{2}})_{c\in\mathcal{S}}
\end{array}
\end{displaymath}

\noindent in the projective space $\mathbb{PR}^{\mathcal{S}}$. We will denote by $\partial_{GM}\mathcal{T}(S):=\text{cl}_{GM}\mathcal{T}(S)\smallsetminus\mathcal{T}(S)$ the \emph{Gardiner--Masur boundary}. Liu--Su have identified the horoboundary of $(\mathcal{T}(S),d_{\mathcal{T}})$ with the Gardiner--Masur boundary \cite{LS14}. There exists an injective map from $\mathcal{PMF}$ to $\partial_{GM}\mathcal{T}(S)$, whose restriction to $\mathcal{PMF}_0$ is a homeomorphism onto its image \cite[Theorem 2]{Miy13}. In particular, the Busemann cocycle $\beta$ on $\partial_{GM}\mathcal{T}(S)$ restricts to a continuous cocycle (again denoted by $\beta$) on $\mathcal{PMF}_0$. Miyachi also proved \cite[Corollary 1]{Miy13} that for all $F\in\mathcal{PMF}_0$, all Teichmüller rays with vertical foliation equal to $F$ converge to $F$ in $\partial_{GM}\mathcal{T}(S)$. For all $F\in\mathcal{PMF}_0$, the horofunction $h_F$ associated to $F$ is given by $$h_F(z)=\log\sup_{\alpha\in\mathcal{S}}\frac{i(F,\alpha)}{\text{Ext}_z(\alpha)^{\frac{1}{2}}}-\log\sup_{\alpha\in\mathcal{S}}\frac{i(F,\alpha)}{\text{Ext}_o(\alpha)^{\frac{1}{2}}}$$ for all $z\in\mathcal{T}(S)$. 

\paragraph*{The curve graph.}\label{sec-cg}

The \emph{curve graph} $\mathcal{C}(S)$ is the simplicial graph whose vertices are the isotopy classes of essential simple closed curves on $S$, in which two vertices are joined by an edge whenever there are disjoint representatives in the isotopy classes of the corresponding curves. We denote by $d_{\mathcal{C}}$ the simplicial metric on $\mathcal{C}(S)$. Masur--Minsky proved in \cite{MM99} that $(\mathcal{C}(S),d_{\mathcal{C}})$ is Gromov hyperbolic, and that an element $\Phi\in\text{Mod}(S)$ acts loxodromically on $\mathcal{C}(S)$ if and only if $\Phi$ is a pseudo-Anosov mapping class. There is a coarsely $\text{Mod}(S)$-equivariant, coarsely Lipschitz map $\pi:\mathcal{T}(S)\to\mathcal{C}(S)$, which sends every point $x\in\mathcal{T}(S)$ to the isotopy class of one of the essential simple closed curves with minimal hyperbolic length in $(S,x)$. Masur--Minsky also proved in \cite{MM99} that $\pi$-images of Teichmüller geodesics are uniform unparameterized quasi-geodesics in $\mathcal{C}(S)$.

The Gromov boundary of the curve graph was identified by Klarreich \cite{Kla99} with the space of equivalence classes of arational foliations, two arational foliations being equivalent if they have the same topological support, and only differ by their transverse measure. Klarreich also proved that there is a well-defined $\text{Mod}(S)$-equivariant map $$\psi:\mathcal{PMF}_0\to\mathcal{C}(S)\cup\partial_{\infty}\mathcal{C}(S),$$ sending any element in $\mathcal{S}$ to the corresponding vertex of $\mathcal{C}(S)$, and such that for all $x\in\mathcal{UE}$ and all sequences $(x_n)_{n\in\mathbb{N}}\in\mathcal{T}(S)^{\mathbb{N}}$ converging to $x$, the sequence $(\pi(x_n))_{n\in\mathbb{N}}$ converges to $\psi(x)\in\partial_{\infty}\mathcal{C}(S)$.

\paragraph*{Random walks on mapping class groups.}

We finish this section by reviewing a result of Kaimanovich--Masur \cite{KM96} about random walks on $\text{Mod}(S)$. A subgroup $H\subseteq\text{Mod}(S)$ is \emph{nonelementary} if it contains two pseudo-Anosov mapping classes that generate a free subgroup of $H$. A probability measure on $\text{Mod}(S)$ is \emph{nonelementary} if the subsemigroup of $\text{Mod}(S)$ generated by the support of $\mu$ is a nonelementary subgroup of $\text{Mod}(S)$. In view of the definition of nonelementarity in Section \ref{sec-hyp}, this is equivalent to nonelementarity with respect to the action on the curve graph $\mathcal{C}(S)$.

\begin{theo}(Kaimanovich--Masur \cite[Theorem 2.2.4]{KM96})\label{unique-stationary}
Let $\mu$ be a nonelementary probability measure on $\text{Mod}(S)$. Then for $\mathbb{P}$-a.e. sample path $\mathbf{\Phi}:=(\Phi_n)_{n\in\mathbb{N}}$ of the random walk on $(\text{Mod}(S),\mu)$, the sequence $(\Phi_n^{-1}.o)_{n\in\mathbb{N}}$ converges to a point $\text{bnd}(\mathbf{\Phi})\in\mathcal{UE}$. The hitting measure $\nu^{\ast}$ on $\mathcal{PMF}$, defined by letting $$\nu^{\ast}(S)=\mathbb{P}[\text{bnd}(\mathbf{\Phi})\in S]$$ for all measurable subsets $S\subseteq\mathcal{PMF}$, is nonatomic, and it is the unique $\check{\mu}$-stationary measure on $\mathcal{PMF}$.
\end{theo}

\subsection{Relating the length cocycle to the Busemann cocycle}

Let $\sigma:\text{Mod}(S)\times\mathcal{S}\to\mathbb{R}$ be the \emph{length cocycle}, defined by letting $$\sigma(\Phi,c):=\log\frac{l_o(\Phi(c))}{l_o(c)}$$ for all $\Phi\in\text{Mod}(S)$ and all $c\in\mathcal{S}$, where $o$ is a fixed basepoint in $\mathcal{C}(S)$. The following proposition will enable us to get control over the length cocycle in terms of the Busemann cocycle. We recall that $h_c$ denotes the horofunction in $\partial_h\mathcal{T}(S)$ associated to $c$.

\begin{prop}\label{Busemann-thick}
For all $\epsilon>0$, there exists $C>0$ such that for all $c\in\mathcal{S}$ and all $z\in\mathcal{T}(S)^{\epsilon}$, one has $$\left|h_c(z)-\log\frac{l_z(c)}{l_o(c)}\right|\le C.$$
\end{prop}

\begin{proof}
It follows from work by Minsky \cite[Lemma 4.3]{Min96} that there exists $C_1>0$ such that for all $z\in\mathcal{T}(S)^{\epsilon}$ and all $c\in\mathcal{S}$, one has $$\left|\log(l_z(c))-\frac{1}{2}\log (\text{Ext}_z(c))\right|\le C_1.$$ Hence for all $z\in\mathcal{T}(S)^{\epsilon}$ and all $c\in\mathcal{S}$, one has $$\left|h_c(z)-\left(\log\sup_{\alpha\in\mathcal{S}}\frac{i(c,\alpha)}{l_z(\alpha)}-\log\sup_{\alpha\in\mathcal{S}}\frac{i(c,\alpha)}{l_o(\alpha)}\right)\right|\le 2C_1.$$ Using Proposition \ref{metric-lengths} and continuity of the extension of the intersection form to the boundary, one can then establish the existence of $C_2>0$ such that for all $z\in\mathcal{T}(S)^{\epsilon}$ and all $c\in\mathcal{S}$, one has $$\left|h_c(z)-\left(\log\sup_{\alpha\in\mu_z}\frac{i(c,\alpha)}{l_z(\alpha)}-\log\sup_{\alpha\in\mu_o}\frac{i(c,\alpha)}{l_o(\alpha)}\right)\right|\le C_2.$$ Notice that for all $x\in\mathcal{T}(S)^{\epsilon}$, all curves in $\mu_x$ coarsely have hyperbolic length $1$, up to a bounded multiplicative error. Therefore, there exists $C_3>0$ such that for all $z\in\mathcal{T}(S)^{\epsilon}$ and all $c\in\mathcal{S}$, one has 
\begin{equation}\label{1eq}
\left|h_c(z)-\log\frac{\sup_{\alpha\in\mu_z}i(c,\alpha)}{\sup_{\alpha'\in\mu_o}i(c,\alpha')}\right|\le C_3.
\end{equation}
 Finally, in view of \cite[Proposition 3.1]{LRT12}, there exists $C_4>0$ such that for all $z\in\mathcal{T}(S)$ and all $c\in\mathcal{S}$, one has $$\frac{1}{C_4}\sum_{\alpha\in\mu_z}i(c,\alpha)l_z(\overline{\alpha})\le l_z(c)\le C_4\sum_{\alpha\in\mu_z}i(c,\alpha)l_z(\overline{\alpha}),$$ where the curve $\overline{\alpha}$ is the transverse curve to the curve $\alpha$ to the short marking $\mu_z$. One derives that $$\frac{1}{C_4}\sup_{\alpha\in\mu_z}(i(c,\alpha)l_z(\overline{\alpha}))\le l_z(c)\le C_4\left(\sum_{\alpha\in\mu_z}l_z(\overline{\alpha})\right)\sup_{\alpha\in\mu_z}i(c,\alpha).$$ It follows from the Collar Lemma that when $z\in\mathcal{T}(S)^{\epsilon}$, all curves in a short marking $\mu_x$ coarsely have hyperbolic length $1$, up to a bounded multiplicative error. One can therefore deduce that there exists $C_5>0$ such that for all $z\in\mathcal{T}(S)^{\epsilon}$ and all $c\in\mathcal{S}$, we have 
\begin{equation}\label{2eq}
\left|\log(l_z(c))-\log\left(\sup_{\alpha\in\mu_z}i(c,\alpha)\right)\right|\le C_5.
\end{equation}
 The claim follows from the estimates \eqref{1eq} and \eqref{2eq}.
\end{proof}

As a consequence of Proposition \ref{Busemann-thick} applied to $z=\Phi^{-1}o$, we obtain the following result.

\begin{cor}\label{cor-key}
There exists $C>0$ such that $$|\beta(\Phi,c)-\sigma(\Phi,c)|\le C$$ for all $\Phi\in\text{Mod}(S)$ and all $c\in\mathcal{S}$.
\qed
\end{cor}

\subsection{A deviation principle in Teichmüller space}\label{sec-rw-teich}

In \cite{Kar14}, Karlsson established a law of large numbers for $\text{Mod}(S)$. This is stated in \cite{Kar14} in the case where $x\in\mathcal{S}$ is a simple closed curve on $S$, however Karlsson's proof extends to the case where $x\in\mathcal{MF}$. We fix once and for all a basepoint $o\in\mathcal{T}(S)$. We recall that $i$ denotes the intersection form on $\mathcal{MF}\times\mathcal{T}(S)$.

\begin{theo}(Karlsson \cite[Corollary 4]{Kar14}) \label{Karlsson}
Let $\mu$ be a nonelementary probability measure on $\text{Mod}(S)$ with finite first moment with respect to $d_{\mathcal{T}}$, and let $\lambda$ be the drift of the random walk on $(\text{Mod}(S),\mu)$ with respect to $d_{\mathcal{T}}$. Then for all $x\in\mathcal{MF}$ and $\mathbb{P}$-a.e. sample path of the random walk on $(\text{Mod}(S),\mu)$, one has $$\lambda=\lim_{n\to +\infty}\frac{1}{n}\log i(\Phi_n.x,o).$$ 
\end{theo}

We deduce the following deviation estimate for the realization in $\mathcal{T}(S)$ of the random walk on $(\text{Mod}(S),\mu)$.

\begin{prop}\label{ccl-mod}
Let $\mu$ be a nonelementary probability measure on $\text{Mod}(S)$ with finite second moment with respect to $d_{\mathcal{T}}$, and let $\lambda$ be the drift of the random walk on $(\text{Mod}(S),\mu)$ with respect to $d_{\mathcal{T}}$. Then for every $\epsilon>0$, there exists a sequence $(C_n)_{n\in\mathbb{N}}\in l^1(\mathbb{N})$ such that for all $n\in\mathbb{N}$ and all $c\in\mathcal{S}$, one has $$\mu^{\ast n}\left(\left\{\Phi\in\text{Mod}(S)|\left|\log\frac{l_{o}(\Phi(c))}{l_{o}(c)}-n\lambda\right|\ge\epsilon n\right\}\right)\le C_n.$$
\end{prop}

\begin{proof}
Let $\sigma:\text{Mod}(S)\times\mathcal{PMF}\to\mathbb{R}$ be the continuous cocycle defined by $$\sigma(\Phi,x):=\log\frac{i(\Phi.x,o)}{i(x,o)}$$ for all $\Phi\in\text{Mod}(S)$ and all $x\in\mathcal{PMF}$. Theorem \ref{Karlsson} states that for all $x\in\mathcal{PMF}$ and $\mathbb{P}$-a.e. sample path of the random walk on $(\text{Mod}(S),\mu)$, one has $$\lim_{n\to +\infty}\frac{1}{n}\sigma(\Phi_n,x)=\lambda.$$ Birkhoff's ergodic theorem then implies that $$\int_{\text{Mod}(S)\times\mathcal{PMF}}\sigma(\Phi,x)d\mu(\Phi)d\nu(x)=\lambda,$$ where $\nu$ denotes the unique $\mu$-stationary probability measure on $\mathcal{PMF}$. Proposition \ref{ccl-mod} follows by applying Benoist--Quint's deviation estimate (Proposition \ref{bq}) to the cocycle $\sigma$ (square-integrability of $\sigma_{sup}$ follows from the fact that $\sigma_{sup}(\Phi,x)= d_{Th}(\Phi.o,o)$ for all $\Phi\in\text{Mod}(S)$).   
\end{proof}

We now establish an analogue of Proposition \ref{ccc} for the realization in $\mathcal{T}(S)$ of the random walk on $(\text{Mod}(S),\mu)$. In the following statement, we let $$\kappa_{\mathcal{T}}(\Phi):=d_{\mathcal{T}}(\Phi.o,o)$$ for all $\Phi\in\text{Mod}(S)$.

\begin{prop}\label{cct-mod}
Let $\mu$ be a nonelementary probability measure on $\text{Mod}(S)$ with finite second moment with respect to $d_{\mathcal{T}}$. Let $\lambda$ be the drift of the random walk on $(\text{Mod}(S),\mu)$ with respect to $d_{\mathcal{T}}$. Then for every $\epsilon>0$, there exists a sequence $(C_n)_{n\in\mathbb{N}}\in l^1(\mathbb{N})$ such that $$\mu^{\ast n}(\{\Phi\in \text{Mod}(S)||\kappa_{\mathcal{T}}(\Phi)-n\lambda|\ge\epsilon n\})\le C_n$$ for all $n\in\mathbb{N}$.
\end{prop}

\begin{proof}
In view of Proposition \ref{metrics}, it is enough to prove the analogous statement where $d_{\mathcal{T}}$ is replaced by $d_{Th}$ in the definition of $\kappa_{\mathcal{T}}$. Proposition \ref{cct-mod} therefore follows from Proposition \ref{ccl-mod} applied to each of the finitely many curves in $\mu_{o}$ given by Proposition \ref{metric-lengths}.
\end{proof}

\subsection{Lifting estimates from $\mathcal{C}(S)$ to $\mathcal{T}(S)$}\label{sec-lift}

\subsubsection{Deviation estimates for the Gromov product: Hypothesis \textbf{(H2)}}

We will now check Hypothesis \textbf{(H2)} from Theorem \ref{tcl-gen} for the Gromov product on the horoboundary of $\mathcal{T}(S)$.

\begin{prop}\label{product-estimate-mod}
Let $\mu$ be a nonelementary probability measure on $\text{Mod}(S)$, and let $\nu^{\ast}$ be the unique $\check{\mu}$-stationary probability measure on $\mathcal{PMF}_0$. Then there exist $\alpha>0$ and a sequence $(C_n)_{n\in\mathbb{N}}\in l^1(\mathbb{N})$ such that for all $x\in{\mathcal{PMF}_0}$, one has $$\nu^{\ast}(\{y\in\mathcal{PMF}_0|(x|y)_{o}\ge \alpha n\})\le C_n.$$
\end{prop}

The strategy of our proof of Proposition \ref{product-estimate-mod} will consist in \emph{lifting} to $\mathcal{T}(S)$ the analogous estimate for the Gromov product on $\mathcal{C}(S)$. In order to make the lifting argument possible, we will appeal to a contraction property for typical Teichmüller geodesics. A similar strategy was already used in \cite{DH15} for proving a spectral theorem for the random walk on $\text{Mod}(S)$. 

Let $K>0$ be a constant such that all $\pi$-images of Teichmüller geodesics are $(K,K)$-unparameterized quasi-geodesics in $\mathcal{C}(S)$. We fix once and for all a large enough constant $\kappa>0$ such that all triangles in $\mathcal{C}(S)$ whose sides are $(K,K)$-quasigeodesics, are $\kappa$-thin (in particular $\kappa$ is assumed to satisfy the conclusion of Proposition \ref{crossing}). We equip $\mathcal{T}(S)$ with the Teichmüller metric $d_{\mathcal{T}}$. Let $I$ be a Teichmüller geodesic segment. Given $B,C>0$, we say that $I$ is \emph{$(B,C)$-progressing} if $\text{diam}_{\mathcal{T}(S)}(I)\le B$ and $\text{diam}_{\mathcal{C}(S)}(\pi(I))\ge C$. Given $D,\tau>0$, we say that $I$ is \emph{$(D,\tau)$-contracting} if for all geodesic segments $J$ in $\mathcal{T}(S)$, if $\pi(I)$ and $\pi(J)$ fellow travel up to distance $\kappa$ in $\mathcal{C}(S)$ (with a slight abuse of terminology, as we are identifying $\pi(I)$ and $\pi(J)$ with their parameterizations), then there exists $J_1\subseteq J$ at $d_{\mathcal{T}}$-Hausdorff distance at most $D$ from $I$, such that $\pi(J_1)$ has diameter at least $\text{diam}(\pi(I))-\tau$ in $\mathcal{C}(S)$. The following proposition, established in \cite{DH15}, essentially follows from work by Dowdall--Duchin--Masur. 

\begin{prop}(Dowdall--Duchin--Masur \cite[Theorem A]{DDM14}, Dahmani--Horbez \cite[Proposition 3.6]{DH15})\label{contraction0}
There exist $C_0,\tau>0$ such that for all $B>0$, there exists $D>0$ such that for all $C>C_0$, all $(B,C)$-progressing Teichmüller geodesic segments are $(D,\tau)$-contracting.
\end{prop}

We fix once and for all the constants $C_0,\tau>0$ given by Proposition \ref{contraction0}. We say that $I$ is \emph{$D$-supercontracting} if for all geodesic segments $J$ in $\mathcal{T}(S)$, if $\pi(I)$ and $\pi(J)$ fellow travel up to distance $\kappa$ in $\mathcal{C}(S)$, then $J$ contains a $(D,\tau)$-contracting subsegment. 

\begin{cor}\label{contraction-mod}
There exists $C>0$ such that for all $B>0$, there exists $D>0$ such that all $(B,C)$-progressing Teichmüller geodesic segments are $D$-supercontracting.
\end{cor}

\begin{proof}
Let $C>C_0+\tau$. If $I$ is $(B,C)$-progressing and $J$ is such that $\pi(I)$ and $\pi(J)$ fellow travel up to distance $\kappa$, then in view of Proposition \ref{contraction0}, there exists a subsegment $J_1\subseteq J$ at $d_{\mathcal{T}}$-Hausdorff distance at most $D$ from $I$, whose $\pi$-image has diameter at least $C-\tau$. In particular, the segment $J_1$ is $(B+2D,C-\tau)$-progressing. Since $C-\tau>C_0$, Proposition \ref{contraction0} applies to $J_1$, showing that $I$ is $D'$-supercontracting for some $D'>0$ only depending on $B$ and $C$.
\end{proof}

From now on, we fix $C>0$ provided by Corollary \ref{contraction-mod}. We now assume that the basepoint in $\mathcal{C}(S)$ is the $\pi$-image of the basepoint in $\mathcal{T}(S)$ (we will denote both of them by $o$). We recall from Section \ref{sec-tracking} that $\tau_{\mathbf{\Phi},\mathcal{C}(S)}$ denotes a $(1,K)$-quasigeodesic ray (where $K$ is a universal constant) from $o$ to the limit point in $\partial_{\infty}\mathcal{C}(S)$ of the sequence $(\Phi_n^{-1}.o)_{n\in\mathbb{N}}$.

\begin{prop}\label{estimate}
Let $\mu$ be a nonelementary probability measure on $\text{Mod}(S)$ with finite second moment with respect to $d_{\mathcal{T}}$.\\
Then there exist constants $B,\beta>0$ and a sequence $(C_n)_{n\in\mathbb{N}}\in l^1(\mathbb{N})$ such that the probability that the Teichmüller segment $[o,\Phi_{n}^{-1}.o]$ contains a $(B,C)$-progressing subsegment whose $\pi$-image in $\mathcal{C}(S)$ fellow travels a subsegment of $\tau_{\mathbf{\Phi},\mathcal{C}(S)}([\beta n,+\infty))$ up to distance $\kappa$, is at least $1-C_n$.
\end{prop}

\begin{proof}
We denote by $\lambda_{\mathcal{C}}$ (resp. $\lambda_{\mathcal{T}}$) the drift of the random walk on $(\text{Mod}(S),\mu)$ with respect to $d_{\mathcal{C}}$ (resp. $d_{\mathcal{T}}$). We let $B$ be a positive real number greater than $5C\lambda_{\mathcal{T}}/\lambda_{\mathcal{C}}$. We denote by $\gamma_{n,\mathbf{\Phi}}:[0,\kappa_{\mathcal{T}}(\Phi_n)]\to\mathcal{T}(S)$ the parameterization of the Teichmüller segment from $o$ to $\Phi_n^{-1}.o$. In view of Propositions \ref{ccc}, \ref{deviation} and \ref{cct-mod}, there exists a sequence $(C_n)_{n\in\mathbb{N}}\in l^1(\mathbb{N})$ such that for all $n\in\mathbb{N}$, with probability at least $1-C_n$, one has $$\kappa_{\mathcal{T}}(\Phi_n)\le \lambda_{\mathcal{T}}n+\frac{\lambda_{\mathcal{C}}Bn}{5C}-B$$ and $$\kappa_{\mathcal{C}}(\Phi_n)>\frac{4\lambda_{\mathcal{C}}n}{5},$$ and, denoting by $t_n^2(\mathbf{\Phi})>0$ the infimum of all real numbers such that $\pi\circ\gamma_{n,\mathbf{\Phi}}([0,t_n^2(\mathbf{\Phi})])$ has $d_{\mathcal{C}}$-diameter at least $\frac{4\lambda_{\mathcal{C}}n}{5}$, then $\pi\circ\gamma_{n,\mathbf{\Phi}}([0,t_n^2(\mathbf{\Phi})])$ fellow travels a subsegment of $\tau_{\mathbf{\Phi},\mathcal{C}(S)}$ up to distance $\kappa$. We claim that in this situation, denoting by $t_n^1(\mathbf{\Phi})>0$ the infimum of all real numbers such that $\pi\circ\gamma_{n,\mathbf{\Phi}}([0,t_n^1(\mathbf{\Phi})])$ has $d_{\mathcal{C}}$-diameter at least $\frac{\lambda_{\mathcal{C}}n}{5}$, the segment $\gamma_{n,\mathbf{\Phi}}([t_n^1(\mathbf{\Phi}),t_n^2(\mathbf{\Phi})])$ contains a $(B,C)$-progressing subsegment whose $\pi$-image fellow travels a subsegment of $\tau_{\mathbf{\Phi},\mathcal{C}(S)}$ up to distance $\kappa$. Proposition \ref{estimate} will follow from this claim (with $\beta:=\frac{\lambda_{\mathcal{C}}}{5}$).

To prove the claim, we subdivide the Teichmüller segment $\gamma_{n,\mathbf{\Phi}}([t_n^1(\mathbf{\Phi}),t_n^2(\mathbf{\Phi})])$ into $\lceil \frac{\kappa_{\mathcal{T}}(\Phi_n)}{B}\rceil$ subsegments of $d_{\mathcal{T}}$-length at most $B$, whose $\pi$-images all fellow travel some subsegment of $\tau_{\mathbf{\Phi},\mathcal{C}(S)}$ up to distance $\kappa$. If none of these segments had a $\pi$-image of diameter at least $C$, then the image $\pi\circ\gamma_{n,\mathbf{\Phi}}([t_n^1(\mathbf{\Phi}),t_n^2(\mathbf{\Phi})])$ would have $d_{\mathcal{C}}$-diameter at most $\left(\frac{\kappa_{\tau}(\Phi_n)}{B}+1\right)C\le\frac{\lambda_{\mathcal{T}}Cn}{B}+\frac{\lambda_{\mathcal{C}}n}{5}\le\frac{2\lambda_{\mathcal{C}}n}{5},$ a contradiction. The claim follows.
\end{proof}

We recall from Theorem \ref{unique-stationary} that for $\mathbb{P}$-a.e. sample path $\mathbf{\Phi}=(\Phi_n)_{n\in\mathbb{N}}$ of the random walk on $(\text{Mod}(S),\mu)$, the sequence $(\Phi_n^{-1}.o)_{n\in\mathbb{N}}$ converges to a point $\text{bnd}_{\mathcal{T}}(\mathbf{\Phi})\in\mathcal{UE}$. We denote by $\tau_{\mathbf{\Phi},\mathcal{T}}$ the Teichmüller ray from $o$ to $\text{bnd}_{\mathcal{T}}(\mathbf{\Phi})$. Notice that the $\pi$-image of $\tau_{\mathbf{\Phi},\mathcal{T}}$ fellow travels $\tau_{\mathbf{\Phi},\mathcal{C}(S)}$ up to distance $\kappa$.

\begin{prop}\label{contraction-typical}
Let $\mu$ be a nonelementary probability measure on $\text{Mod}(S)$ with finite second moment with respect to $d_{\mathcal{T}}$.\\
Then there exist constants $D,\alpha,\beta>0$, and a sequence $(C_n)_{n\in\mathbb{N}}\in l^1(\mathbb{N})$, such that the probability that $\tau_{\mathbf{\Phi},\mathcal{T}}([0,\alpha n])$ contains a $(D,\tau)$-contracting subsegment whose $\pi$-image lies at $d_{\mathcal{C}}$-distance at least $\beta n$ from $o$, is at least $1-C_n$.
\end{prop}

\begin{proof}
Let $B,\beta>0$ be the constants given by Proposition \ref{estimate}. Let $D>0$ be the constant corresponding to $B$ provided by Corollary \ref{contraction-mod}. Let $\lambda$ be the drift of the random walk on $(\text{Mod}(S),\mu)$ with respect to $d_{\mathcal{T}}$. Propositions \ref{cct-mod} and \ref{estimate} imply that there exists a sequence $(C_n)_{n\in\mathbb{N}}\in l^1(\mathbb{N})$ such that with probability at least $1-C_n$, the Teichmüller segment $[o,\Phi_n^{-1}.o]$ has length at most $2\lambda n$, and contains a subsegment $I$ whose $\pi$-image lies at $d_{\mathcal{C}}$-distance at least $\beta n$ from $o$, which is $(B,C)$-progressing, and such that $\pi(I)$ fellow travels the $\pi$-image of a subsegment $J$ of $\tau_{\mathbf{\Phi},\mathcal{T}}$ up to distance $\kappa$. In view of Corollary \ref{contraction-mod}, the segment $I$ is $D$-supercontracting. This implies that $J$ contains a $(D,\tau)$-contracting subsegment.
\end{proof}

\begin{proof}[Proof of Proposition \ref{product-estimate-mod}]
Let $D,\alpha,\beta>0$ and $(C_n)_{n\in\mathbb{N}}\in l^1(\mathbb{N})$ be as in Proposition \ref{contraction-typical}. Let $x\in\mathcal{PMF}_0$. We will show that $$\mathbb{P}\left[(x|\text{bnd}_{\mathcal{T}}(\mathbf{\Phi}))_{o}\le 2\alpha n\right]\ge 1-C_n$$ for all $n\in\mathbb{N}$. 

Using nonatomicity of the hitting measure $\nu^{\ast}$ and Proposition \ref{contraction-typical}, we get the existence a measurable subset $X$ of the path space, of measure at least $1-C_n$, such that for all $\mathbf{\Phi}\in X$, we have $\text{bnd}_{\mathcal{T}}(\mathbf{\Phi})\in\mathcal{PMF}_0\smallsetminus\{x\}$, and the segment $\tau_{\mathbf{\Phi},\mathcal{T}}([0,\alpha n])$ contains a $(D,\tau)$-contracting subsegment $J$ such that $\pi(J)$ lies at $d_{\mathcal{C}}$-distance at least $\beta n$ from $o$. In view of Proposition \ref{crossing}, we can also assume that for all $\mathbf{\Phi}\in X$, the Teichmüller geodesic line (or ray) from $x$ to $\text{bnd}_{\mathcal{T}}(\mathbf{\Phi})$ contains a subsegment whose $\pi$-image fellow travels $\pi(J)$ up to distance $\kappa$. One can thus find a point $y\in [x,\text{bnd}_{\mathcal{T}}(\mathbf{\Phi})]$ and a point $y'\in\tau_{\mathbf{\Phi},\mathcal{T}}([0,\alpha n])$, such that $d_{\mathcal{T}}(y,y')\le D$. From now on, we let $\mathbf{\Phi}\in X$.

Let $\gamma:\mathbb{R}\to\mathcal{T}(S)$ be a parameterization of the Teichmüller line from $x$ to $\text{bnd}_{\mathcal{T}}(\mathbf{\Phi})$. Then $$(x|\text{bnd}_{\mathcal{T}}(\mathbf{\Phi}))_o=-\frac{1}{2}\inf_{z\in\mathcal{T}(S)}\lim_{n\to +\infty}(h_{\gamma(-n)}(z)+h_{\gamma(n)}(z)),$$ with the notations from Section \ref{sec-horo}. It then follows from the triangle inequality that the infimum in the above formula is achieved at any point lying on the image of $\gamma$. In particular, one has $$(x|\text{bnd}_{\mathcal{T}}(\mathbf{\Phi}))_o=-\frac{1}{2}(h_x(y)+h_{\text{bnd}_{\mathcal{T}}(\mathbf{\Phi})}(y)).$$ Since $\tau_{\mathbf{\Phi}}(k)$ also converges to $\text{bnd}_{\mathcal{T}}(\mathbf{\Phi})$, using Miyachi's result \cite[Corollary 1]{Miy13} and the identification between the horofunction boundary and the Gardiner--Masur boundary, we have (all limits are taken as $k$ goes to $+\infty$, and we write $d$ instead of $d_{\mathcal{T}}$ for ease of notation):
\begin{displaymath}
\begin{array}{rl}
-2(x|\text{bnd}_{\mathcal{T}}(\mathbf{\Phi}))_{o}&=\lim [d(y,\gamma(-k))+d(y,\tau_{\mathbf{\Phi},\mathcal{T}}(k))-d(o,\gamma(-k))-d(o,\tau_{\mathbf{\Phi},\mathcal{T}}(k))]\\
&\ge \lim [d(y,\gamma(-k))+(d(y',\tau_{\mathbf{\Phi},\mathcal{T}}(k))-D)-d(o,\gamma(-k))-d(o,\tau_{\mathbf{\Phi},\mathcal{T}}(k))]\\
&=\lim [d(y,\gamma(-k))-d(o,\gamma(-k))]-d(o,y')-D\\
&\ge -d(o,y)-d(o,y')-D\\
&\ge -2d(o,y')-2D\\
&\ge -2\alpha n-2D.
\end{array}
\end{displaymath}
\noindent This implies that $$(x|\text{bnd}_{\mathcal{T}}(\mathbf{\Phi}))_{o}\le \alpha n+D$$ and concludes the proof of Proposition \ref{product-estimate-mod}.
\end{proof}

\subsubsection{Mean value of the Busemann cocycle: Hypothesis \textbf{(H1)}}

Using similar arguments as in the proof of Proposition \ref{product-estimate-mod}, we will now establish Hypothesis \textbf{(H1)} from Theorem \ref{tcl-gen} for the Gromov product on the horoboundary of $\mathcal{T}(S)$: this is the content of Corollary \ref{h1-mod} below.

\begin{prop}\label{product-estimate-2-mod}
Let $\mu$ be a nonelementary probability measure on $\text{Mod}(S)$. For all $\epsilon>0$, there exists $M>0$ such that for all $x\in{\mathcal{PMF}_0}$, one has $$\mathbb{P}\left[\sup_{n\in\mathbb{N}}|\beta(\Phi_n,x)-\kappa(\Phi_n)|\le M\right]\ge 1-\epsilon.$$
\end{prop}

\begin{proof}
It suffices to show that for all $x\in\mathcal{PMF}_0$ and $\mathbb{P}$-a.e. sample path $(\Phi_n)_{n\in\mathbb{N}}$ of the random walk on $(\text{Mod}(S),\mu)$, one has $$\sup_{n\in\mathbb{N}}|\beta(\Phi_n,x)-\kappa(\Phi_n)|<+\infty.$$ We observe that there exists $D>0$ such that for all $x\in\mathcal{PMF}_0$ and $\mathbb{P}$-a.e. sample path $\mathbf{\Phi}:=(\Phi_n)_{n\in\mathbb{N}}$ of the random walk, the $\pi$-image of the Teichmüller geodesic from $x$ to $\text{bnd}_{\mathcal{T}}(\mathbf{\Phi})$ crosses the $\pi$-image of a $(D,\tau)$-contracting subsegment $I$ of $\tau_{\mathbf{\Phi},\mathcal{T}}$ up to distance $\kappa$. This observation relies on the fact that $\mathbb{P}$-almost surely, the Teichmüller ray from $o$ to $\text{bnd}_{\mathcal{T}}(\mathbf{\Phi})$ contains infinitely many $(D,\tau)$-contracting subsegments: this fact was established in \cite[Proposition 3.11]{DH15}. In particular, there exists a point $z\in I$ such that for all $n\in\mathbb{N}$ sufficiently large, both the Teichmüller segment $[o,\Phi_n^{-1}.o]$ and the Teichmüller ray from $\Phi_n^{-1}.o$ to $x$ pass at bounded distance from $z$. So for all $n\in\mathbb{N}$ sufficiently large, the difference $|\beta(\Phi_n,x)-\kappa(\Phi_n)|$ is equal to $|h_x(z)-d_{\mathcal{T}}(o,z)|$ up to a bounded error (by a similar computation as in the proof of Proposition \ref{product-estimate-mod}), from which the claim follows.
\end{proof}

\begin{cor}\label{h1-mod}
Let $\mu$ be a nonelementary probability measure on $\text{Mod}(S)$ with finite first moment with respect to $d_{\mathcal{T}}$. Let $\lambda$ be the drift of the random walk on $(\text{Mod}(S),\mu)$ with respect to $d_{\mathcal{T}}$. Let $\nu$ be the $\mu$-stationary probability measure on $\mathcal{PMF}_0$. Then $$\int_{\text{Mod}(S)\times\mathcal{PMF}_0}\beta(\Phi,y)d\mu(\Phi)d\nu(y)=\lambda.$$
\end{cor}

\begin{proof}
Proposition \ref{product-estimate-2-mod} implies that for all $y\in\mathcal{PMF}_0$ and $\mathbb{P}$-a.e. sample path $(\Phi_n)_{n\in\mathbb{N}}$ of the random walk on $(\text{Mod}(S),\mu)$, one has $$\lim_{n\to +\infty}\frac{1}{n}\beta(\Phi_n,y)=\lambda.$$ Corollary \ref{h1-mod} then follows by applying Birkhoff's ergodic theorem.
\end{proof}

\subsection{Central limit theorem}\label{sec-tcl}

We now complete the proof of the central limit theorem for mapping class groups.

\begin{theo}\label{tcl-mod}
Let $S$ be a closed, connected, oriented, hyperbolic surface, and let $\rho$ be a hyperbolic metric on $S$. Let $\mu$ be a nonelementary probability measure on $\text{Mod}(S)$ with finite second moment with respect to the Teichmüller metric. Let $\lambda$ be the drift of the random walk on $(\text{Mod}(S),\mu)$ with respect to $d_{\mathcal{T}}$.\\ 
Then there exists a centered Gaussian law $N_{\mu}$ on $\mathbb{R}$ such that for every compactly supported continuous function $F$ on $\mathbb{R}$, and all essential simple closed curves $c$ on $S$, one has $$\lim_{n\to +\infty}\int_{\text{Mod}(S)} F \left(\frac{\log l_{\rho}(\Phi(c))-n\lambda}{\sqrt{n}}\right)d\mu^{\ast n}(\Phi) = \int_{\mathbb{R}}F(t)dN_{\mu}(t),$$ uniformly in $c$.
\end{theo}

\begin{proof}
In view of Corollary \ref{cor-key}, it is enough to prove a central limit theorem for the Busemann cocycle $\beta$, i.e. show that there exists a gaussian law $N_{\mu}$ on $\mathbb{R}$ such that for every compactly supported continuous function $F$ on $\mathbb{R}$, and all $x\in\mathcal{PMF}_0$, one has $$\lim_{n\to +\infty}\int_{\text{Mod}(S)} F \left(\frac{\beta(\Phi,x)-n\lambda}{\sqrt{n}}\right)d\mu^{\ast n}(\Phi) = \int_{\mathbb{R}}F(t)dN_{\mu}(t),$$ uniformly in $x\in\mathcal{PMF}_0$. In view of Proposition \ref{product-estimate-2-mod}, it is enough to prove that there exists $x\in{\mathcal{PMF}_0}$ for which the limit holds. Since the Busemann cocycle $\beta$ satisfies Hypotheses \textbf{(H1)} (Corollary \ref{h1-mod}, applied to the nonelementary probability measure $\check{\mu}$) and \textbf{(H2)} (Proposition \ref{product-estimate-mod}) from Theorem \ref{tcl-gen} (with $Y^-=Y^+=\mathcal{PMF}_0$), the result follows from Theorem \ref{tcl-gen}.
\end{proof}

\section{Central limit theorem on $\text{Out}(F_N)$}\label{sec-out-1}

Let $N\ge 2$. The goal of this section is to establish a central limit theorem on $\text{Out}(F_N)$ (Theorem \ref{tcl-intro}). The proof will follow the same outline as in the mapping class group case -- the main novelties coming from the need to take care of asymmetry of the metric on $CV_N$, which implies in particular that we will only have a \emph{one-sided} version of Proposition \ref{contraction-mod} in the context of $\text{Out}(F_N)$.

\subsection{Background on $\text{Out}(F_N)$}\label{sec-out}

\paragraph*{Outer space and its metric.}

\noindent \emph{Outer space} $CV_N$ was introduced by Culler--Vogtmann in \cite{CV86}, and can be defined as the space of equivalence classes of free, minimal, simplicial, isometric $F_N$-actions on simplicial metric trees, two trees being equivalent whenever there exists an $F_N$-equivariant homothety between them. \emph{Unprojectivized outer space} $cv_N$ is defined in a similar way, by considering trees up to $F_N$-equivariant isometry, instead of homothety. The group $\text{Out}(F_N)$ acts on both $CV_N$ and $cv_N$ on the right by precomposing the $F_N$-actions. These $\text{Out}(F_N)$-actions can be turned into left actions by letting $\Phi.T:=T.\Phi^{-1}$ for all $\Phi\in\text{Out}(F_N)$ and all $T\in CV_N$. Given $\epsilon>0$, the \emph{$\epsilon$-thick part} $CV_N^{\epsilon}$ is the subspace of $CV_N$ made of those trees $T$ such that the volume one representative of the quotient graph $T/F_N$ does not contain any embedded loop of length smaller than $\epsilon$.

Outer space comes equipped with a natural asymmetric metric $d_{CV_N}$ \cite{FM11}, the distance between two trees $T,T'\in CV_N$ being equal to the logarithm of the infimal Lipschitz constant of an $F_N$-equivariant map from the covolume one representative of $T$, to the covolume one representative of $T'$. The $\text{Out}(F_N)$-action on $CV_N$ is by isometries for this metric. White has proved (see \cite[Proposition 3.15]{FM11} or \cite[Proposition 2.3]{AK11}) that $$d_{CV_N}(T,T')=\log\sup_{g\in F_N\smallsetminus\{e\}}\frac{||g||_{T'}}{||g||_T}$$ for all $T,T'\in CV_N$, identified with their covolume one representatives in the above formula. In addition, the supremum in the above formula can be taken over a finite set $\text{Cand}(T)$ that only depends on $T$. Elements in $\text{Cand}(T)$ are called \emph{candidates} for $T$, they are primitive elements of $F_N$ (recall that an element of $F_N$ is \emph{primitive} if it belongs to some free basis of $F_N$).

\paragraph*{Currents on free groups.}

Let $\partial^2F_N:=\partial F_N\times\partial F_N\smallsetminus\Delta$, where $\partial F_N$ is identified with the Gromov boundary of a Cayley tree $R_0$ of $F_N$, and $\Delta$ denotes the diagonal subset. A \emph{current} on $F_N$ is an $F_N$-invariant Borel measure on $\partial^2F_N$ that is finite on compact subsets of $\partial^2F_N$. We denote by $\text{Curr}_N$ the space of currents on $F_N$, which is topologized as in \cite{}. Every $g\in F_N$ which is not of the form $h^k$ for any $h\in F_N$ and $k>1$ determines a \emph{rational current} $\eta_g$, where for all closed-open subsets $S\subseteq\partial^2F_N$, the number $\eta_g(S)$ is the number of $F_N$-translates of the axis of $g$ in $R_0$, whose pairs of endpoints belong to $S$. The group $\text{Out}(F_N)$ acts on the set of currents on the left in the following way: given $\Phi\in\text{Out}(F_N)$, a current $\eta$, and a compact set $K\subseteq\partial^2F_N$, we let $\Phi(\eta)(K):=\eta(\phi^{-1}(K))$, where $\phi\in\text{Aut}(F_N)$ is any representative of $\Phi$. The length pairing between trees in $cv_N$ and elements of $F_N$ extends continuously \cite{KL09} to an intersection pairing $\langle .,.\rangle : \overline{cv_N}\times Curr_N\to\mathbb{R}_+$.

\paragraph*{Forward and backward horoboundaries of outer space.}

We denote by $\mathcal{P}_N$ the collection of all primitive elements of $F_N$. The \emph{primitive compactification} $\overline{CV_N}^{prim}$ was introduced in \cite[Section 2.4]{Hor14-1} by taking the closure of the image of the embedding
\begin{displaymath}
\begin{array}{cccc}
i:&CV_N &\to & \mathbb{PR}^{\mathcal{P}_N}\\
& g&\mapsto & \mathbb{R}^{\ast}(||g||_T)_{g\in \mathcal{P}_N}
\end{array}
\end{displaymath}
in the projective space $\mathbb{PR}^{\mathcal{P}_N}$. The forward horofunction compactification of $CV_N$ was identified in \cite[Theorem 2.2]{Hor14} with the primitive compactification $\overline{CV_N}^{prim}$. This is a quotient of Culler--Morgan's compactification $\overline{CV_N}$, which was introduced in \cite{CM87}, and identified by Cohen--Lustig \cite{CL95} and Bestvina--Feighn \cite{BF94} (see also \cite{Hor14-5}) with the space of equivariant homothety classes of \emph{very small} minimal $F_N$-trees, i.e. trees whose arc stabilizers are either trivial, or maximally cyclic, and whose tripod stabilizers are trivial. The fibers of the quotient map from $\overline{CV_N}$ to $\overline{CV_N}^{prim}$ were described in \cite{Hor14-1}.

Among trees in $\partial CV_N:=\overline{CV_N}\smallsetminus CV_N$, \emph{arational trees} will be of particular interest to us. These are defined in the following way. A subgroup $A\subseteq F_N$ is a \emph{free factor} if there exists $B\subseteq F_N$ such that $F_N=A\ast B$. A tree $T\in\partial{CV_N}$ is \emph{arational} if no proper free factor of $F_N$ has a global fixed point in $T$, and the action of every proper free factor of $F_N$ on its minimal subtree in $T$ is free and simplicial. We denote by $\mathcal{UE}$ the subspace of $\partial{CV_N}$ made of those arational trees $T$ which are both \emph{uniquely ergometric} and \emph{dually uniquely ergodic}, i.e. those that admit, up to homothety, a unique length measure (see the definition in \cite[Section 5.1]{Gui00}, attributed to Paulin) and a unique geodesic current $\eta\in M_N$ satisfying $\langle T,\eta\rangle=0$ (we say that $\eta$ is \emph{dual} to $T$). We note that the quotient map from $\overline{CV_N}$ to $\overline{CV_N}^{prim}$ is one-to-one in restriction to the set of trees with dense orbits \cite{Hor14-1}, and in particular in restriction to $\mathcal{UE}$.

Properties of the backward horoboundary $\partial_h^-CV_N$ were also investigated in \cite[Section 4]{Hor14}. Backward horofunctions are described in terms of geodesic currents on $F_N$. We denote by $M_N$ the minimal set of currents, defined in \cite{Mar97} as the closure of the set of rational currents associated to primitive conjugacy classes. Given a finite set $S\subseteq M_N$, we define a function $f_S$ on $CV_N$ by setting $$f_S(T):=\log\frac{\sup_S\langle T,\eta\rangle}{\sup_S\langle o,\eta\rangle}$$ for all $T\in CV_N$ (here again trees are identified with their covolume $1$ representatives). By \cite[Proposition 4.5]{Hor14}, for all $\xi\in\partial_h^-CV_N$, there exists a finite set $S\subseteq M_N$ such that $\xi=f_S$. For all trees $T\in\mathcal{UE}$ with dual current $\eta$, and all geodesic lines $\gamma:\mathbb{R}\to CV_N$ such that $\lim_{t\to -\infty}\gamma(t)=T$ in Culler--Morgan's compactification $\overline{CV_N}$, one has $\lim_{t\to -\infty}\gamma(t)=f_{[\eta]}$ in $CV_N\cup\partial_h^-CV_N$, see \cite[Remark 4.6]{Hor14}.

\paragraph*{Folding lines in outer space.}

A nice collection of paths in outer space is the collection of so-called \emph{folding lines}, whose definition we now review. A \emph{morphism} between two $\mathbb{R}$-trees $T$ and $T'$ is a map $f:T\to T'$, such that every segment in $T$ can be subdivided into finitely many subsegments, in restriction to which $f$ is an isometry. Every morphism defines a partition of the set of connected components of $T\smallsetminus\{x\}$ (called \emph{directions}), at each point $x\in T$: two directions belong to the same class of the partition if and only if their $f$-images overlap in $T'$. The data of all these partitions is called a \emph{train-track structure} on $T$. A morphism $f:T\to T'$ is \emph{optimal} if there are at least two distinct equivalence classes of directions at every point in $T$, and $f$ realizes the infimal Lipschitz constant of an $F_N$-equivariant map from $T$ to $T'$.

An \emph{optimal folding path} in $cv_N$ is a continuous map $\gamma:I\to cv_N$, where $I\subseteq\mathbb{R}$ is an interval, together with a collection of $F_N$-equivariant optimal morphisms $f_{t,t'}:\gamma(t)\to \gamma(t')$ for all $t<t'$, such that $f_{t,t''}=f_{t',t''}\circ f_{t,t'}$ for all $t<t'<t''$. It is a \emph{greedy folding path} if for all $t_0\in\mathbb{R}$, there exists $\epsilon>0$ such that for all $t\in [t_0,t_0+\epsilon]$, the tree $\gamma(t)$ is obtained from $\gamma(t_0)$ by identifying two segments of length $\epsilon$ in $\gamma(t_0)$ whenever they have a common endpoint, and have the same $f_{t_0,t}$-image. The projection to $CV_N$ of a (greedy) folding path in $cv_N$ will again be called a (greedy) folding path. 

Any two trees $T,T'\in CV_N$ are joined by a (non-unique) geodesic segment, which is the concatenation of a segment contained in a \emph{simplex} of $CV_N$ (i.e. the subspace of $CV_N$ made of all trees obtained by only varying the edge lengths of $T$), and an optimal greedy folding path. A geodesic segment obtained in this way will be called a \emph{standard geodesic segment}. If $T'\in\mathcal{UE}$, then one can similarly find a (non-unique) \emph{standard geodesic ray} in $CV_N$ starting at $o$ and limiting at $T'$, consisting of the concatenation of an initial segment contained in a simplex, and an optimal greedy folding ray, see \cite[Lemma 6.11]{BR13}. Given any two distinct trees $T,T'\in\mathcal{UE}$, there exists a (non-unique) optimal greedy folding line $\gamma:\mathbb{R}\to CV_N$ such that $\gamma(t)$ converges to $T$ (resp. to $T'$) as $t$ goes to $-\infty$ (resp. $+\infty$), as follows from the work of Bestvina--Reynolds \cite[Theorem 6.6 and Lemma 6.11]{BR13}.

\paragraph*{The free factor graph.}

The \emph{free factor graph} $FF_N$ is the simplicial graph whose vertices are the conjugacy classes of proper free factors of $F_N$, in which two vertices $[A]$ and $[B]$ are joined by an edge whenever there exist representatives $A,B$ in the corresponding conjugacy classes, such that either $A\varsubsetneq B$ or $B\varsubsetneq A$. The graph $FF_N$ is Gromov hyperbolic (Bestvina--Feighn \cite{BF12}). The group $\text{Out}(F_N)$ has a natural left action on $FF_N$. There is a natural coarsely Lipschitz, coarsely $\text{Out}(F_N)$-equivariant map $\pi:CV_N\to FF_N$, which sends any tree $T\in CV_N$ to a proper free factor $A$ such that $T$ collapses to a tree $T'$ in which $A$ fixes a point. Bestvina--Feighn also established in \cite{BF12} that $\pi$-images of standard geodesic lines in $CV_N$ are uniform unparameterized quasi-geodesics in $FF_N$.

The Gromov boundary $\partial_{\infty}FF_N$ was described independently by Bestvina--Reynolds \cite{BR13} and Hamenstädt \cite{Ham12} as the space of equivalence classes of arational trees in $\partial CV_N$, two trees being equivalent whenever they have the same underlying topological tree (and only differ by the metric). In particular, there is a continuous $\text{Out}(F_N)$-equivariant map $\psi:\mathcal{UE}\to\partial_{\infty}FF_N$, such that for all $T\in\mathcal{UE}$, and all sequences $(S_n)_{n\in\mathbb{N}}\in CV_N^{\mathbb{N}}$ converging to $T$ (for the topology of $\overline{CV_N}$), the sequence $(\pi(S_n))_{n\in\mathbb{N}}$ converges to $\psi(T)$ (for the topology of $FF_N\cup\partial_{\infty}FF_N$).

\paragraph*{Review on projections to folding paths.}

Let $\gamma:I\to CV_N$ (where $I\subseteq\mathbb{R}$ is an interval) be an optimal greedy folding path determined by a morphism $f$. The morphism $f$ determines train-track structures on all trees $\gamma(t)$ with $t\in I$. A segment $[a,b]\subseteq\gamma(t)$ is \emph{legal} if for every $x\in J$, the intervals $[a,x)$ and $(x,b]$ belong to components of $\gamma(t)\smallsetminus\{x\}$ in distinct equivalence classes of the train-track structure. Following Bestvina--Feighn \cite{BF12}, for all $g\in\mathcal{P}_N$, we define $\text{right}_{\gamma}(g)$ as the infimal $t\in I$ such that every segment of length $M_{BF}$ in the axis of $g$ in $\gamma(t)$, contains a legal subsegment of length $3$ (here $M_{BF}$ is the constant defined in \cite[Section 6]{BF12}, which only depends on $N$). We then let $$\text{Pr}_{\gamma}(S):=\gamma\left(\sup_{g\in\text{Cand}(S)}\text{right}_{\gamma}(g)\right)$$ for all $S\in CV_N$. The following lemma relies on Bestvina--Feighn's work \cite[Section 4]{BF12}, it was established as such in \cite[Lemma 4.7]{DH15}.

\begin{lemma}(Bestvina--Feighn \cite{BF12})\label{bf}
There exists $K_0>0$ (only depending on $N$) such that for all greedy folding paths $\gamma:I\to CV_N$, all $g\in \mathcal{P}_N$, and all $t,t'\ge\text{right}_{\gamma}(g)$ satisfying $t\le t'$, one has $$\left|d_{CV_N}(\gamma(t),\gamma(t'))-\log\frac{||g||_{\gamma(t')}}{||g||_{\gamma(t)}}\right|\le K_0.$$
\end{lemma}

The Bestvina--Feighn projection satisfies the following contraction property.

\begin{lemma}(Bestvina--Feighn \cite[Proposition 7.2]{BF12})\label{BF-contr}
There exists $D_1>0$ such that for every standard geodesic line $\gamma:I\to CV_N$, and all $H,H'\in CV_N$, if $d_{CV_N}(H,H')\le d_{CV_N}(H,\text{Im}(\gamma))$, then $d_{FF_N}(\pi(\text{Pr}_{\gamma}(H)),\pi(\text{Pr}_{\gamma}(H')))\le D_1$.
\end{lemma}

The following lemma of Dowdall--Taylor, relates the Bestvina--Feighn projection and the closest-point projection in $FF_N$. Given an optimal greedy folding path $\gamma:I\to CV_N$, we denote by $\mathbf{n}_{\pi\circ\gamma}$ a closest-point projection map to the image of $\pi\circ\gamma$ in $FF_N$.

\begin{lemma}(Dowdall--Taylor \cite[Lemma 4.2]{DT14})\label{DT}
There exists $D_2>0$ such that $$d_{FF_N}(\pi(\text{Pr}_{\gamma}(H)),\mathbf{n}_{\pi\circ\gamma}(\pi(H)))\le D_2$$ for all optimal greedy folding paths $\gamma$ and all $H\in CV_N$.
\end{lemma}

\paragraph*{Random walks on $\text{Out}(F_N)$.}

A subgroup $H\subseteq\text{Out}(F_N)$ is \emph{nonelementary} if $H$ is not virtually cyclic, and does not virtually fix the conjugacy class of any proper free factor of $F_N$. A probability measure on $\text{Out}(F_N)$ is \emph{nonelementary} if the subsemigroup generated by its support is a nonelementary subgroup of $\text{Out}(F_N)$. With the terminology from Section \ref{sec-hyp}, this is equivalent to nonelementarity with respect to the action on the free factor graph $FF_N$.

\begin{prop}(Namazi--Pettet--Reynolds \cite[Theorem 7.21]{NPR14})\label{NPR}
Let $\mu$ be a nonelementary probability measure on $\text{Out}(F_N)$, with finite first moment with respect to $d_{CV_N}$. Then for $\mathbb{P}$-a.e. sample path $\mathbf{\Phi}:=(\Phi_n)_{n\in\mathbb{N}}$ of the random walk on $(\text{Out}(F_N),\mu)$, and any $o\in CV_N$, the sequence $(\Phi_n^{-1}.o)_{n\in\mathbb{N}}$ converges to a point $\text{bnd}(\mathbf{\Phi})\in\mathcal{UE}$. The hitting measure $\nu^{\ast}$ defined by setting $$\nu^{\ast}(S):=\mathbb{P}(\text{bnd}(\mathbf{\Phi})\in S)$$ for all measurable subsets $S\subseteq \partial CV_N$ is nonatomic, and it is the unique $\check{\mu}$-stationary probability measure on $\partial CV_N$. 
\end{prop}

It follows from the description of $\partial_h^+CV_N$ as a quotient of Culler--Morgan's boundary that $\nu^{\ast}$ can also be viewed as the unique $\check{\mu}$-stationary probability measure on $\partial_h^+CV_N$.

\subsection{Progress and contraction for folding paths}

In this section, we will establish a contraction property for folding lines in outer space (Proposition \ref{contraction} below, which is a variation on \cite[Proposition 4.17]{DH15}). This will play the same role in our proof of the central limit theorem as Proposition \ref{contraction-mod} in the mapping class group case, though we only get a \emph{one-sided} version in the $\text{Out}(F_N)$ context.

Let $\kappa>0$ be a sufficiently large constant, such that all quasi-geodesic triangles in $FF_N$ whose sides are $\pi$-images of folding lines in outer space, are $\kappa$-thin (in particular $\kappa$ is assumed to satisfy the conclusion of Proposition \ref{crossing}, if $K$ is a constant such that $\pi$-images of folding lines are $(K,K)$-unparameterized quasigeodesics). Given $D>0$, a geodesic line $\gamma:\mathbb{R}\to CV_N$ and a subsegment $I=[a,b]\subseteq\mathbb{R}$, we say that $\gamma$ is \emph{$D$-bicontracting along $I$} if for all geodesic segments $\gamma':[a',b']\to CV_N$, if $\pi\circ\gamma'$ contains a subsegment which fellow travels $\pi\circ\gamma_{|I}$ up to distance $\kappa$, then there exists $t\in [a',b']$ such that $d_{CV_N}^{sym}(\gamma(a),\gamma'(t))\le D$. We say that $\gamma$ is \emph{$D$-right-contracting along $I$} if the above holds for all geodesic segments $\gamma'$ which are further assumed to be such that $\gamma'(b')$ belongs to the image of $\gamma$ and lies to the right of $\gamma_{|I}$. 

Given $B,C>0$, we say that a geodesic segment $\gamma:I=[a,b]\to CV_N$ is \emph{$(B,C)$-progressing} if $d_{CV_N}(\gamma(a),\gamma(b))\le B$ and $\pi\circ\gamma(I)$ has $d_{FF_N}$-diameter at least $C$. The following lemma is based on an observation due to Dowdall--Taylor \cite[Lemma 4.3]{DT14}. We include a proof for completeness.

\begin{lemma}(Dowdall--Taylor \cite[Lemma 4.3]{DT14})\label{thick}
There exists $C_0>0$ such that for all $B>0$, there exists $\epsilon>0$ such that for all $C>C_0$, if $\gamma:[a,b]\to CV_N$ is a $(B,C)$-progressing geodesic segment, then $\gamma(a)\in CV_N^{\epsilon}$.
\end{lemma}

\begin{proof}
Let $\epsilon<\exp(-B)$, and assume by contradiction that $\gamma(a)\notin CV_N^{\epsilon}$: there exists $g\in F_N$ represented by a loop of length smaller than $\epsilon$ in the volume one representative of $\gamma(a)/F_N$. Let $t\in [a,b]$ be such that $d_{FF_N}(\pi\circ\gamma(a),\pi\circ\gamma(t))\ge 11$ (this exists as soon as $C_0$ is sufficiently large). Then $||g||_{\gamma(t)}\ge 1$ (see the argument in \cite[Lemma 4.2]{DT14} for the precise constant $11$). Hence $$B\ge d_{CV_N}(\gamma(a),\gamma(t))\ge\log\frac{1}{\epsilon}$$ (the first inequality follows from the fact that $\gamma$ is $(B,C)$-progressing, and the second follows from White's formula for the distance on $CV_N$). This is a contradiction.
\end{proof}

\begin{prop}\label{contraction}
There exists $C>0$ such that for all $B>0$, there exists $D>0$ such that the following holds.\\
Let $\gamma:\mathbb{R}_+\to CV_N$ be a geodesic ray. Let $I\subseteq\mathbb{R}$ be an interval such that $\gamma_{|I}$ is a $(B,C)$-progressing optimal greedy folding path. Then $\gamma$ is $D$-right-contracting along $I$.
\end{prop}

The proof of Proposition \ref{contraction} is a variation on the argument from \cite[Proposition 4.17]{DH15}: in \cite{DH15}, the folding path $\gamma$ was supposed to satisfy a stronger condition, but the conclusion was that $\gamma$ is $D$-bicontracting along $I$. 

\begin{proof}
The proof is illustrated on Figure \ref{fig}. We let $I:=[a,b]$. Let $C\ge 2\kappa+D_1+2D_2+C_0$, where $D_1,D_2,C_0>0$ are the constants from Lemmas \ref{BF-contr}, \ref{DT} and \ref{thick}, respectively. Let $B>0$. Assume that $\gamma_{|I}$ is $(B,C)$-progressing. Let $\gamma':[a',b']\to CV_N$ be a geodesic segment such that $\gamma'(b')\in\gamma([b,+\infty))$, and there exists $I'\subseteq\ [a',b']$ such that $\pi\circ\gamma'_{|I'}$ fellow travels $\pi\circ\gamma_{|I}$ up to distance $\kappa$. We aim at showing that there exists $t\in I'$ such that $d_{CV_N}^{sym}(\gamma(a),\gamma'(t))\le D$, where $D$ is a constant that only depends on $B$. 

Let $S:=\gamma(a)$. As $C\ge C_0$ and $\gamma_{|I}$ is $(B,C)$-progressing, we have $S\in CV_N^{\epsilon}$, where $\epsilon:=\epsilon(B)$ is the constant provided by Lemma \ref{thick}. Hence by \cite{AKB12}, there exists $M>0$ (only depending on $B$) such that 
\begin{equation}\label{akb}
d_{CV_N}^{sym}(S,y)\le M d_{CV_N}(y,S)
\end{equation}
for all $y\in CV_N$.

As $\pi\circ\gamma'_{|I'}$ fellow travels $\pi\circ\gamma_{|I}$ up to distance $\kappa$, there exists $U\in\gamma'(I')$ such that $d_{FF_N}(\pi(U),\pi(S))\le \kappa$. Therefore $d_{FF_N}(\mathbf{n}_{\pi\circ\gamma_{|I}}(\pi(U)),\pi(S))\le 2\kappa$, and by Lemma \ref{DT} we have 
\begin{equation}\label{eq+}
d_{FF_N}(\pi(\text{Pr}_{I}(U)),\pi(S))\le 2\kappa+D_2
\end{equation}
(here we are using $\text{Pr}_I$ as a shortcut for $\text{Pr}_{\gamma_{|I}}$). 

Let $U'\in\gamma'(I)$ be a point lying to the right of $U$ on the image of $\gamma'$, and let $S_0\in\gamma(I)$, be such that $$d_{CV_N}(U,U')=d_{CV_N}(U,\gamma(I))=d_{CV_N}(U,S_0).$$ Lemma \ref{BF-contr} shows that $d_{FF_N}(\pi(\text{Pr}_I(U)),\pi(\text{Pr}_I(S_0)))\le D_1$. By Lemma \ref{DT}, we also have $d_{FF_N}(\pi(S_0),\pi(\text{Pr}_I(S_0)))\le D_2$. Together with Equation~\eqref{eq+}, the triangle inequality then yields $d_{FF_N}(\pi(S_0),\pi(S))\le 2\kappa+D_1+2D_2$. Recall that $S,S_0\in\gamma(I)$, and $S_0$ lies to the right of $S$. As $C\ge 2\kappa+D_1+2D_2$ and $\gamma_{|I}$ is $(B,C)$-progressing, this implies that $d_{CV_N}(S,S_0)\le B$, and in view of Equation \eqref{akb} we obtain  
\begin{equation}\label{eqe0}
d_{CV_N}^{sym}(S,S_0)\le MB.
\end{equation}
Lemma \ref{BF-contr} also shows that $d_{FF_N}(\pi(\text{Pr}_I(U')),\pi(\text{Pr}_I(U)))\le D_1$. Together with Equation~\eqref{eq+}, the triangle inequality then yields $d_{FF_N}(\pi(\text{Pr}_I(U')),\pi(S))\le 2\kappa+D_1+D_2$. This implies as above that $d_{CV_N}(S,\text{Pr}_I(U'))\le B$, and hence
\begin{equation}\label{e1}
d_{CV_N}^{sym}(S,\text{Pr}_I(U'))\le MB
\end{equation} 
in view of Equation \eqref{akb}.
\\
\\
Let now $g\in \mathcal{P}_N$ be such that 
\begin{equation}\label{1}
d_{CV_N}(U',\text{Pr}_{I}(U'))=\log\frac{||g||_{\text{Pr}_{I}(U')}}{||g||_{U'}}.
\end{equation}
Let $S':=\gamma'(b')\in\gamma([b,+\infty))$. Notice that $\text{Pr}_I(U')$ is necessarily to the right of $\text{Pr}_{\gamma}(U')$ on the image of $\gamma$ (they may coincide), and $S'$ lies to the right of $\text{Pr}_I(U')$. Lemma \ref{bf}, applied to the optimal greedy folding path $\gamma$, shows that there exists $K_0>0$ (which only depends on $N$) such that 
\begin{equation}\label{2}
d_{CV_N}(\text{Pr}_{I}(U'),S')\le\log\frac{||g||_{S'}}{||g||_{\text{Pr}_{I}(U')}}+K_0.
\end{equation}
By adding Equations \eqref{1} and \eqref{2}, we get
\begin{equation}\label{eq3}
d_{CV_N}(U',\text{Pr}_{I}(U'))+d_{CV_N}(\text{Pr}_{I}(U'),S')-K_0\le\log\frac{||g||_{S'}}{||g||_{U'}}\le d_{CV_N}(U',S').
\end{equation}
Equation \eqref{e1} then shows that there exists $K>0$ (which only depends on $B$ and on the rank $N$ of the free group) such that 
\begin{equation}\label{e3}
d_{CV_N}(U',S)+d_{CV_N}(S,S')-K\le d_{CV_N}(U',S').
\end{equation} 
Therefore, the total green length on Figure \ref{fig} is equal, up to a bounded error, to $d_{CV_N}(U,S')$. This implies that $d_{CV_N}(U',S)$ is uniformly bounded. Precisely, we have
\begin{displaymath}
\begin{array}{rl}
d_{CV_N}(U',S)&\le d_{CV_N}(U',S')-d_{CV_N}(S,S')+K\\
&\le d_{CV_N}(U',S')-d_{CV_N}(U,S')+d_{CV_N}(U,S)+K\\
&=-d_{CV_N}(U,U')+d_{CV_N}(U,S)+K\\
&\le -d_{CV_N}(U,U')+d_{CV_N}(U,S_0)+d_{CV_N}(S_0,S)+K\\
&=d_{CV_N}(S_0,S)+K\\
&\le MB+K,
\end{array}
\end{displaymath}
and therefore $$d_{CV_N}^{sym}(U',S)\le M(MB+K)$$ as required, in view of Equation \eqref{akb}. 
\end{proof}

\begin{figure}
\begin{center}
\input{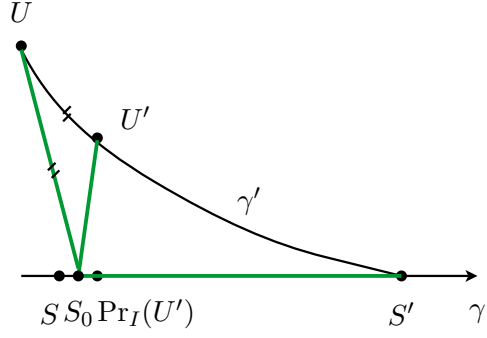}
\caption{The situation in the proof of Proposition \ref{contraction}.}
\label{fig}
\end{center}
\end{figure}

\subsection{A deviation principle in outer space}

Let $\mu$ be a probability measure on $\text{Out}(F_N)$ with finite first moment with respect to $d_{CV_N}$. We recall that the \emph{drift} of the random walk on $(\text{Out}(F_N),\mu)$ with respect to $d_{CV_N}$ is defined with the conventions of the present paper as being equal to the limit $$\lim_{n\to +\infty}\frac{1}{n}d_{CV_N}(\Phi_n.o,o)$$ for $\mathbb{P}$-a.e. sample path $(\Phi_n)_{n\in\mathbb{N}}$ of the left random walk on $(\text{Out}(F_N),\mu)$.

\begin{theo}(Horbez \cite[Corollary 5.4]{Hor14})\label{LLN}
Let $\mu$ be a nonelementary probability measure on $\text{Out}(F_N)$ with finite first moment with respect to $d_{CV_N}$, and let $\lambda$ be the drift of the random walk on $(\text{Out}(F_N),\mu)$ with respect to $d_{CV_N}$.\\ Then for all $x\in\partial_h^-CV_N$ and $\mathbb{P}$-a.e. sample path $(\Phi_n)_{n\in\mathbb{N}}$ of the random walk on $(\text{Out}(F_N),\mu)$, one has $$\lim_{n\to +\infty}\frac{1}{n}\beta^-(\Phi_n,x)=\lambda.$$ 
\end{theo}

\begin{proof}
With the notations from Section \ref{sec-out}, we have $x=f_S$ for some finite subset $S\subseteq M_N$, and $$\beta^-(\Phi_n,x)=f_S(\Phi_n^{-1}.o)=\log\frac{\sup_S\langle o,\Phi_n\eta\rangle}{\sup_S\langle o,\eta\rangle}.$$ The conclusion of Theorem \ref{LLN} was established in \cite[Corollary 5.4]{Hor14} when $S$ consists of a single rational current, corresponding to a primitive conjugacy class in $F_N$. However, the proof from \cite{Hor14} extends in the same way to any singleton $S$, by noticing that the supremum in the formula giving the distance between two trees in $CV_N$ can be taken over all currents in $M_N$, if one uses the continuous extension of the length pairing between conjugacy classes and trees to $M_N\times \overline{cv_N}$. The conclusion then also holds for any finite subset $S\subseteq M_N$.
\end{proof}

We derive the following large deviation principle for the Busemann cocycle on the backward horoboundary of outer space, and for the function $\kappa_{CV_N}$ defined by letting $$\kappa_{CV_N}(\Phi):=d_{CV_N}(\Phi.o,o)$$ for all $\Phi\in\text{Out}(F_N)$.

\begin{prop}\label{ccl}
Let $\mu$ be a nonelementary probability measure on $\text{Out}(F_N)$ with finite second moment with respect to $d_{CV_N}$, and let $\lambda$ be the drift of the random walk on $(\text{Out}(F_N),\mu)$ with respect to $d_{CV_N}$. Then for every $\epsilon>0$, there exists a sequence $(C_n)_{n\in\mathbb{N}}\in l^1(\mathbb{N})$ such that for all $n\in\mathbb{N}$ and all $x\in\partial_h^-CV_N$, one has $$\mu^{\ast n}\left(\{\Phi\in\text{Out}(F_N)||\beta^-(\Phi,x)-n\lambda|\ge\epsilon n\}\right)\le C_n.$$ In particular, one has $$\mu^{\ast n}(\{\Phi\in\text{Out}(F_N)||\log||\Phi(g)||-n\lambda|\ge\epsilon n\})\le C_n$$ for all $n\in\mathbb{N}$ and all $g\in\mathcal{P}_N$, and $$\mu^{\ast n}(\{\Phi\in \text{Out}(F_N)||\kappa_{CV_N}(\Phi)-n\lambda|\ge\epsilon n\})\le C_n.$$
\end{prop}

\begin{proof}
Theorem \ref{LLN}, together with Birkhoff's ergodic theorem, implies that $$\int_{\text{Out}(F_N)\times\partial_h^-CV_N}\beta^-(\Phi,x)d\mu(\Phi)d\nu(x)=\lambda$$ for every $\mu$-stationary probability measure $\nu$ on $\partial_h^-CV_N$. The first part of Proposition \ref{ccl} then follows from Proposition \ref{bq} (the fact that $\beta^-_{sup}\in L^2(G,\mu)$ follows from Equation~\eqref{bus} in Section \ref{sec-bus}). The second assertion is a specification of the first to the case where $x=f_S$, with $S$ consisting of a single rational current corresponding to the primitive element $g$. The last assertion follows from the second, applied to each of the finitely many candidates in $\text{Cand}(o)$.
\end{proof}

\subsection{Lifting estimates from $FF_N$ to $CV_N$}

\subsubsection{Deviation estimate for the Gromov product: Hypothesis \textbf{(H2)}}

We will now check Hypothesis \textbf{(H2)} from Theorem \ref{tcl-gen} for the Gromov product on $\partial_h^-CV_N\times\partial_h^+CV_N$. The proof follows the same outline as in the mapping class group case (Section \ref{sec-lift}).

\begin{prop}\label{product-estimate}
Let $\mu$ be a nonelementary probability measure on $\text{Out}(F_N)$ with finite second moment with respect to $d_{CV_N}$, and let $\nu^{\ast}$ be the unique $\check{\mu}$-stationary probability measure on $\partial_h^+CV_N$. Then there exist $\alpha>0$ and a sequence $(C_n)_{n\in\mathbb{N}}\in l^1(\mathbb{N})$ such that for all $x\in\mathcal{UE}\subseteq\partial_h^-CV_N$, one has $$\nu^{\ast}(\{y\in\partial_h^+CV_N|(x|y)_o\ge\alpha n\})\le C_n.$$
\end{prop}

We assume that the basepoint in $FF_N$ is the $\pi$-image of the basepoint in $CV_N$ (both are denoted by $o$). Recall from Proposition \ref{NPR} that for $\mathbb{P}$-a.e. sample path $\mathbf{\Phi}:=(\Phi_n)_{n\in\mathbb{N}}$ of the random walk on $(\text{Out}(F_N),\mu)$, the sequence $(\Phi_n^{-1}.o)_{n\in\mathbb{N}}$ converges to a point $\text{bnd}(\mathbf{\Phi})\in\mathcal{UE}$. We then let $\tau_{\mathbf{\Phi}}$ be a standard geodesic ray from $o$ to $\text{bnd}(\mathbf{\Phi})$, and for all $n\in\mathbb{N}$, we let $t_n(\mathbf{\Phi})\in\mathbb{R}_+$ be the infimum of all $t\in\mathbb{R}_+$ such that $$d_{CV_N}(\Phi_n^{-1}.o,\tau_{\mathbf{\Phi}}(t))=\inf_{t'\in\mathbb{R}_+}d_{CV_N}(\Phi_n^{-1}.o,\tau_{\mathbf{\Phi}}(t')).$$ Notice that the image of $\pi\circ\tau_{\mathbf{\Phi}}$ in $FF_N$ lies at bounded Hausdorff distance from any $(1,K)$-quasigeodesic ray from $o$ to $\text{bnd}_{FF_N}(\mathbf{\Phi})$. We start by establishing the following lemma.

\begin{lemma}\label{lemma-progress}
Let $\mu$ be a nonelementary probability measure on $\text{Out}(F_N)$ with finite second moment with respect to $d_{CV_N}$. Then there exist $K_1,K_2>0$ and a sequence $(C_n)_{n\in\mathbb{N}}\in l^1(\mathbb{N})$ such that $$\mathbb{P}[d_{CV_N}(o,\tau_{\mathbf{\Phi}}(t_n(\mathbf{\Phi})))\le K_1n \text{~~\emph{and}~~} d_{FF_N}(o,\pi\circ\tau_{\mathbf{\Phi}}(t_n(\mathbf{\Phi})))\ge K_2n]\ge 1-C_n.$$
\end{lemma}

\begin{proof}
Let $\lambda$ be the drift of the random walk on $(\text{Out}(F_N),\mu)$ with respect to $d_{CV_N}$. In view of Proposition \ref{ccl}, there exists a sequence $(C_n)_{n\in\mathbb{N}}\in l^1(\mathbb{N})$ such that for all $n\in\mathbb{N}$, there exists a measurable subspace $X_n$ of the path space with $\mathbb{P}(X_n)\ge 1-C_n$, and such that $d_{CV_N}(o,\Phi_n^{-1}.o)\le 2\lambda n$ for all $\mathbf{\Phi}:=(\Phi_k)_{k\in\mathbb{N}}\in X_n$. For all $\mathbf{\Phi}\in X_n$, since $o$ and $\Phi_n^{-1}.o$ both belong to the thick part of $CV_N$, we have $$d_{CV_N}(\Phi_n^{-1}.o,o)\le 2M\lambda n$$ for some constant $M>0$ only depending on the rank $N$ of the free group \cite{AKB12}. Since $\tau_{\mathbf{\Phi}}(t_n(\mathbf{\Phi}))$ is a closest-point projection of $\Phi_n^{-1}.o$ to the image of $\tau_{\mathbf{\Phi}}$, we then have $$d_{CV_N}(\Phi_n^{-1}.o,\tau_{\mathbf{\Phi}}(t_n(\mathbf{\Phi})))\le 2M\lambda n,$$ and the triangle inequality implies that $$d_{CV_N}(o,\tau_{\mathbf{\Phi}}(t_n(\mathbf{\Phi})))\le 2(M+1)\lambda n.$$ In view of Lemmas \ref{BF-contr} and \ref{DT}, there also exists $K>0$, only depending on the rank $N$ of the free group, such that $$d_{FF_N}(\pi\circ\tau_{\mathbf{\Phi}}(t_n(\mathbf{\Phi})),\mathbf{n}_{\pi\circ\tau_{\mathbf{\Phi}}}(\pi(\Phi_n^{-1}.o)))\le K.$$ On the other hand, Propositions \ref{ccc} and \ref{deviation} imply that we can find a constant $K_2>0$, a sequence $(C'_n)_{n\in\mathbb{N}}\in l^1(\mathbb{N})$, and for all $n\in\mathbb{N}$, a measurable subset $X'_n$ of the path space with $\mathbb{P}(X'_n)\ge 1-C'_n$, and such that $$d_{FF_N}(o,\mathbf{n}_{\pi\circ\tau_{\mathbf{\Phi}}}(\pi(\Phi_n^{-1}.o)))\ge K_2n$$ for all $\mathbf{\Phi}\in X'_n$. Then all sample paths $\mathbf{\Phi}\in X_n\cap X'_n$ satisfy the required estimates, so the lemma follows.
\end{proof}

\begin{prop}\label{prop-contr}
Let $\mu$ be a nonelementary probability measure on $\text{Out}(F_N)$, with finite second moment with respect to $d_{CV_N}$. Then there exist $K_1,D,\epsilon,\beta>0$ and a sequence $(C_n)_{n\in\mathbb{N}}\in l^1(\mathbb{N})$ such that with probability at least $1-C_n$, there exists a subsegment $I:=[a,b]\subseteq [0,K_1n]$ such that $\tau_{\mathbf{\Phi}}(a)\in CV_N^{\epsilon}$, and $\pi\circ\tau_{\mathbf{\Phi}}(I)$ lies at distance at least $\beta n$ from $o$ in $FF_N$, and $\tau_{\mathbf{\Phi}}$ is $D$-right-contracting along $I$.
\end{prop}

\begin{proof}
The proof is similar to the proof of Proposition \ref{estimate}. Let $C$ be a constant that satisfies the conclusions of Lemma \ref{thick} and Proposition \ref{contraction}. Let $K_1,K_2>0$ be the constants provided by Lemma \ref{lemma-progress}, and let $B$ be a positive real number greater than $4CK_1/K_2$. By Lemma \ref{lemma-progress}, there exists a sequence $(C_n)_{n\in\mathbb{N}}\in l^1(\mathbb{N})$ such that for all $n\in\mathbb{N}$, with probability at least $1-C_n$, one has 
\begin{equation}\label{eqe1}
d_{CV_N}(o,\tau_{\mathbf{\Phi}}(t_n(\mathbf{\Phi})))\le K_1n
\end{equation}
and 
\begin{equation}\label{eqe2}
d_{FF_N}(o,\pi\circ\tau_{\mathbf{\Phi}}(t_n(\mathbf{\Phi})))\ge K_2n.
\end{equation}
We claim that when \eqref{eqe1} and \eqref{eqe2} hold, assuming in addition that $n\ge\frac{4C}{K_2}$, if we denote by $t_n^0(\mathbf{\Phi})>0$ the smallest real number such that $\pi\circ\tau_{\mathbf{\Phi}}([0,t_n^0(\mathbf{\Phi})])$ has $d_{FF_N}$-diameter at least $\frac{K_2n}{4}$, then the segment $\tau_{\mathbf{\Phi}}([t_n^0(\mathbf{\Phi}),t_n(\mathbf{\Phi})])$ contains a $(B,C)$-progressing subsegment. Proposition \ref{prop-contr} will follow from this claim (with $\beta:=\frac{K_2}{4}$), together with Lemma \ref{thick}, which states that progressing subsegments are thick, and Proposition \ref{contraction}, which states that progressing subsegments are right-contracting.

To prove the claim, we subdivide $\tau_{\mathbf{\Phi}}([t_n^0(\mathbf{\Phi}),t_n(\mathbf{\Phi})])$ into $\lceil \frac{K_1n}{B}\rceil$ subsegments of $d_{CV_N}$-length at most $B$. If none of these segments had a $\pi$-image of diameter at least $C$, then the image $\pi\circ\tau_{\mathbf{\Phi}}([t_n^0(\mathbf{\Phi}),t_n(\mathbf{\Phi})])$ would have $d_{FF_N}$-diameter at most $\left(\frac{K_1n}{B}+1\right)C$, which is smaller than $\frac{K_2n}{2}$ since $n\ge\frac{4C}{K_2}$. The image $\pi\circ\tau_{\mathbf{\Phi}}([0,t_n(\mathbf{\Phi})])$ would then have diameter at most $\frac{3K_2n}{4}$, which is a contradiction. The claim follows.
\end{proof}

\begin{proof}[Proof of Proposition \ref{product-estimate}]
The proof is similar to the proof of Proposition \ref{product-estimate-mod}. Let $K_1,D,\epsilon,\beta>0$ and $(C_n)_{n\in\mathbb{N}}\in l^1(\mathbb{N})$ be as in Proposition \ref{prop-contr}. Let $x\in\mathcal{UE}$. We will show that there exists $\alpha'>0$ such that $$\mathbb{P}\left[(x|\text{bnd}(\mathbf{\Phi}))_{o}\le \alpha' n\right]\ge 1-C_n$$ for all $n\in\mathbb{N}$. 

Let $n\in\mathbb{N}$. By Proposition \ref{prop-contr}, and since the hitting measure $\nu^{\ast}$ on $\partial_h^+CV_N$ is nonatomic, there exists a measurable subset $X_n$ of the path space, of measure at least $1-C_n$, such that for all $\mathbf{\Phi}\in X_n$, we have $\text{bnd}(\mathbf{\Phi})\in\mathcal{UE}\smallsetminus\{x\}$, and there exists a subsegment $I:=[a,b]\subseteq [0,K_1n]$, such that $\tau_{\mathbf{\Phi}}(a)\in CV_N^{\epsilon}$, and $\pi\circ\tau_{\mathbf{\Phi}}(I)$ lies at $d_{FF_N}$-distance at least $\beta n$ from $o$, and $\tau_{\mathbf{\Phi}}$ is $D$-right-contracting along $I$. In view of Proposition \ref{crossing} (and the existence of the map $\psi:\mathcal{UE}\to\partial_{\infty}FF_N$), we can also assume that for all $\mathbf{\Phi}\in X$, and all sequences $(x_k)_{k\in\mathbb{N}}\in CV_N^{\mathbb{N}}$ converging to $x$, any geodesic segment from $x_k$ to $\tau_{\mathbf{\Phi}}(k)$ with $k\in\mathbb{N}$ sufficiently large contains a subsegment whose $\pi$-image fellow travels $\pi\circ\tau_{\mathbf{\Phi}}(I)$ up to distance $\kappa$. 

From now on, we let $\mathbf{\Phi}\in X_n$. Let $(x_k)_{k\in\mathbb{N}}\in CV_N^{\mathbb{N}}$ be a sequence that converges to $x$. For all sufficiently large $k\in\mathbb{N}$, there exists a point $y_k\in CV_N$ on a standard geodesic segment $\gamma_k$ from $x_k$ to $\tau_{\mathbf{\Phi}}(k)$ satisfying $d_{CV_N}^{sym}(y_k,\tau_{\mathbf{\Phi}}(a))\le D$. The segments $\gamma_k$ then accumulate \cite[Theorem 6.6]{BR13} to an optimal greedy folding line $\gamma:\mathbb{R}\to CV_N$ from $x$ to $\text{bnd}(\mathbf{\Phi})$. There exists a point $y$ lying on the image of $\gamma$ (obtained as an accumulation point of the points $y_k$) such that $d_{CV_N}^{sym}(y,\tau_{\mathbf{\Phi}}(a))\le D$.  We have $$(x|\text{bnd}(\mathbf{\Phi}))_o=-\frac{1}{2}\inf_{z\in CV_N}\lim_{n\to +\infty}(h_{\gamma(-n)}(z)+h_{\gamma(n)}(z)),$$ with the notations from Section \ref{sec-horo}. It then follows from the triangle inequality that the infimum in the above formula is achieved at any point $z$ lying on the image of $\gamma$. In particular, one has $$(x|\text{bnd}(\mathbf{\Phi}))_o=-\frac{1}{2}(h_x(y)+h_{\text{bnd}(\mathbf{\Phi})}(y)).$$

Since $\tau_{\mathbf{\Phi}}(k)$ also converges to $\text{bnd}(\mathbf{\Phi})$ as $k$ goes to $+\infty$, we have (all limits are taken as $k$ goes to $+\infty$, and we use $d$ to denote $d_{CV_N}$, and $d^{sym}$ to denote $d_{CV_N}^{sym}$, for ease of notation):
\begin{displaymath}
\begin{array}{rl}
-2(x|\text{bnd}(\mathbf{\Phi}))_{o}&=\lim [d(\gamma(-k),y)+d(y,\tau_{\mathbf{\Phi}}(k))-d(\gamma(-k),o)-d(o,\tau_{\mathbf{\Phi}}(k))]\\
&\ge \lim [d(\gamma(-k),y)+(d(\tau_{\mathbf{\Phi}}(a),\tau_{\mathbf{\Phi}}(k))-D)-d(\gamma(-k),o)-d(o,\tau_{\mathbf{\Phi}}(k))]\\
&=\lim [d(\gamma(-k),y)-d(\gamma(-k),o)]-d(o,\tau_{\mathbf{\Phi}}(a))-D\\
&\ge -d(y,o)-d(o,\tau_{\mathbf{\Phi}}(a))-D\\
&\ge -2d^{sym}(o,\tau_{\mathbf{\Phi}}(a))-2D.
\end{array}
\end{displaymath}
\noindent Since $\tau_{\mathbf{\Phi}}(a)\in CV_N^{\epsilon}$, there exists \cite{AKB12} a constant $M>0$ such that $d_{CV_N}^{sym}(o,\tau_{\mathbf{\Phi}}(a))\le M d_{CV_N}(o,\tau_{\mathbf{\Phi}}(a))$. This implies that $$(x|\text{bnd}(\mathbf{\Phi}))_{o}\le MK_1 n+D$$ and concludes the proof of Proposition \ref{product-estimate}.
\end{proof}

\subsubsection{Mean value of $\beta^+$: Hypothesis \textbf{(H1)}}

We now establish Hypothesis \textbf{(H1)} from Theorem \ref{tcl-gen} for the cocycle $\beta^+$: this is the content of Corollary \ref{average} below.

\begin{prop}\label{product-estimate-2}
Let $\mu$ be a nonelementary probability measure on $\text{Out}(F_N)$ with finite first moment with respect to $d_{CV_N}$, and let $\nu$ be the $\mu$-stationary probability measure on $\partial_h^+CV_N$. Then there exists a measurable subset $Y\subseteq\partial_h^+CV_N$ with $\nu(Y)=1$, such that for all $y\in Y$ and all $\epsilon>0$, there exists $C>0$ such that $$\mathbb{P}\left[\sup_{n\in\mathbb{N}}|\beta^+(\Phi_n,y)-d_{CV_N}(\Phi_n^{-1}.o,o)|\le C\right]\ge 1-\epsilon.$$ 
\end{prop}

\begin{proof}
Let $Y\subseteq\mathcal{UE}$ be a subset of $\partial_h^+CV_N$ of $\nu$-measure $1$ such that for all $y\in Y$, any standard geodesic ray from $o$ to $y$ contains infinitely many $D$-right-contracting subsegments (the existence of $Y$ was established in \cite[Proposition 4.25]{DH15}). It suffices to show that for all $y\in Y$ and $\mathbb{P}$-a.e. sample path $(\Phi_n)_{n\in\mathbb{N}}$ of the random walk on $(\text{Out}(F_N),\mu)$, one has $$\sup_{n\in\mathbb{N}}|\beta^+(\Phi_n,y)-d_{CV_N}(\Phi_n^{-1}.o,o)|<+\infty.$$ 

Let $y\in Y$. For $\mathbb{P}$-a.e. sample path $(\Phi_n)_{n\in\mathbb{N}}$ of the random walk on $(\text{Out}(F_N),\mu)$, there exists a $D$-right-contracting subsegment $I$ on a standard geodesic ray $\tau$ from $o$ to $y$, and $z\in I$, such that for all sufficiently large $n,k\in\mathbb{N}$, any geodesic ray from $\Phi_n^{-1}.o$ to $\tau(k)$ passes at $d_{CV_N}^{sym}$-distance at most $D$ from $z$. For all $n\in\mathbb{N}$, we have (we write $d$ instead of $d_{CV_N}$ for ease of notation): $$|\beta^+(\Phi_n,y)-d(\Phi_n^{-1}.o,o)|=\lim_{k\to +\infty}|d(\Phi_n^{-1}.o,\tau(k))-d(o,\tau(k))-d(\Phi_n^{-1}.o,o)|.$$ This is equal up to a bounded error to $|d_{CV_N}(\Phi_n^{-1}.o,z)-d_{CV_N}(o,z)-d_{CV_N}(\Phi_n^{-1}.o,o)|$, which is bounded above by $2d_{CV_N}^{sym}(o,z)$ by the triangle inequality.
\end{proof}

\begin{cor}\label{average}
Let $\mu$ be a nonelementary probability measure on $\text{Out}(F_N)$, with finite first moment with respect to $d_{CV_N}$. Let $\check{\lambda}$ be the drift of the random walk on $(\text{Out}(F_N),\check{\mu})$ with respect to $d_{CV_N}$, and let $\nu$ be the unique $\mu$-stationary probability measure on $\partial_h^+CV_N$.\\ Then $$\int_{\text{Out}(F_N)}\int_{\partial_h^+CV_N}\beta^+(g,y)d\mu(g)d\nu(y)=\check{\lambda}.$$  
\end{cor}

\begin{proof}
We first notice that by Kingman's subadditive ergodic theorem, for $\mathbb{P}$-a.e. sample path $(\Phi_n)_{n\in\mathbb{N}}$ of the random walk on $(\text{Out}(F_N),\mu)$, the limit $$\lim_{n\to +\infty}\frac{1}{n}d_{CV_N}(\Phi_n^{-1}.o,o)$$ exists, and is equal to $$\inf_{n\in\mathbb{N}}\int_{\text{Out}(F_N)}d(\Phi^{-1}.o,o)d\mu^{\ast n}(\Phi),$$ which is nothing but the drift $\check{\lambda}$. Proposition \ref{product-estimate-2} therefore implies that for $\nu$-a.e. $y\in\partial_h^+CV_N$ and $\mathbb{P}$-a.e. sample path $(\Phi_n)_{n\in\mathbb{N}}$ of the random walk on $(\text{Out}(F_N),\mu)$, one has $$\lim_{n\to +\infty}\frac{1}{n}\beta^+(\Phi_n,y)=\check{\lambda}.$$ Corollary \ref{average} then follows from Birkhoff's ergodic theorem.
\end{proof}

\subsection{Relating $\beta^-$ and $\kappa_{CV_N}$ and the length cocycle}

We will relate the Busemann cocycle $\beta^-$, the function $\kappa_{CV_N}$ and the length cocycle, in order to reduce the proof of the central limit theorem (Theorem \ref{tcl-intro}) to proving a central limit theorem for $\beta^-$. The following proposition is a \emph{dual} version of Proposition \ref{product-estimate-2} for the cocycle $\beta^-$. 

\begin{prop}\label{dual}
Let $\mu$ be a nonelementary probability measure on $\text{Out}(F_N)$ with finite first moment with respect to $d_{CV_N}$. Then for all $x\in\mathcal{UE}$ and all $\epsilon>0$, there exists $C>0$ such that $$\mathbb{P}\left[\sup_{n\in\mathbb{N}}|\beta^-(\Phi_n,x)-\kappa(\Phi_n)|\le C\right]\ge 1-\epsilon.$$ 
\end{prop}

\begin{proof}
It suffices to prove that for all $x\in\mathcal{UE}$ and $\mathbb{P}$-a.e. sample path $\mathbf{\Phi}:=(\Phi_n)_{n\in\mathbb{N}}$ of the random walk on $(\text{Out}(F_N),\mu)$, one has $$\sup_{n\in\mathbb{N}}|\beta^-(\Phi_n,x)-\kappa(\Phi_n)|<+\infty.$$ For $\mathbb{P}$-a.e. sample path $\mathbf{\Phi}$ of the random walk, the standard geodesic ray $\tau_{\mathbf{\Phi}}$ contains infinitely many $D$-bicontracting subsegments \cite[Proposition 4.25]{DH15}. Let $(x_k)_{k\in\mathbb{N}}\in CV_N^{\mathbb{N}}$ be a sequence that converges to $x$. We have $$|\beta^-(\Phi_n,x)-\kappa(\Phi_n)|=\lim_{k\to +\infty}|d_{CV_N}(x_k,\Phi_n^{-1}.o)-d_{CV_N}(x_k,o)-d_{CV_N}(o,\Phi_n^{-1}.o)|.$$ By the bicontraction property, there exists $z\in\tau_{\mathbf{\Phi}}(\mathbb{R}_+)$ such that for all sufficiently large $k,n\in\mathbb{N}$, any standard geodesic segment from $x_k$ to $\Phi_n^{-1}.o$ passes at bounded $d_{CV_N}^{sym}$-distance from $z$, and similarly any standard geodesic segment from $o$ to $\Phi_n^{-1}.o$ passes at bounded $d_{CV_N}^{sym}$-distance from $z$. Therefore $|\beta^-(\Phi_n,x)-\kappa(\Phi_n)|$ is equal up to a bounded error to $$\lim_{k\to +\infty}|d_{CV_N}(x_k,z)-d_{CV_N}(x_k,o)-d_{CV_N}(o,z)|=|h_x^-(z)-d_{CV_N}(o,z)|,$$ which concludes the proof of Proposition \ref{dual}. 
\end{proof}

The following proposition is a version of Proposition \ref{dual} above in the case where $x$ is a rational current associated to a primitive conjugacy class, instead of $x\in\mathcal{UE}$.

\begin{prop}\label{norm-busemann}
Let $\mu$ be a nonelementary probability measure on $\text{Out}(F_N)$, with finite first moment with respect to $d_{CV_N}$. Then for all $g\in \mathcal{P}_N$ and all $\epsilon>0$, there exists $C>0$ such that $$\mathbb{P}\left[\sup_{n\in\mathbb{N}}\left|\kappa(\Phi_n)-\log\frac{||\Phi_n(g)||_{o}}{||g||_o}\right|\le C\right]\ge 1-\epsilon.$$ 
\end{prop}

\begin{proof}
It is enough to prove that for $\mathbb{P}$-a.e. sample path $\mathbf{\Phi}:=(\Phi_n)_{n\in\mathbb{N}}$ of the random walk on $(\text{Out}(F_N),\mu)$, we have $$\sup_{n\in\mathbb{N}}\left|\kappa(\Phi_n)-\log\frac{||\Phi_n(g)||_{o}}{||g||_o}\right|<+\infty.$$ By \cite[Proposition 4.25]{DH15}, for $\mathbb{P}$-a.e. sample path $\mathbf{\Phi}:=(\Phi_n)_{n\in\mathbb{N}}$ of the random walk, the standard geodesic ray $\tau_{\mathbf{\Phi}}$ contains infinitely many $D$-bicontracting subsegments. Let $g\in\mathcal{P}_N$, and let $I$ be a $D$-bicontracting subsegment of $\tau_{\mathbf{\Phi}}$ lying to the right of $\tau_{\mathbf{\Phi}}(\text{Pr}_{\tau_{\mathbf{\Phi}}}(g))$. 

For all $n\in\mathbb{N}$, let $\beta_n$ be a standard geodesic segment from $o$ to $\Phi_n^{-1}.o$. Then there exists $z'\in I$ such that for all sufficiently large $n\in\mathbb{N}$, the segment $\beta_n$ contains a point $z_n$ at bounded $d_{CV_N}^{sym}$-distance of $z'$. In addition, Lemma \ref{bf} applied to $\beta_n$ implies that for all $n\in\mathbb{N}$, we have $$\left|d_{CV_N}(z_n,\Phi_n^{-1}.o)-\log\frac{||\Phi_n(g)||_o}{||g||_{z_n}}\right|\le K.$$ Since all points $z_n$ lie at the same $d_{CV_N}^{sym}$-distance from $o$ (up to a bounded additive error), there exists $K'>0$ such that $$\left|d_{CV_N}(o,\Phi_n^{-1}.o)-\log\frac{||\Phi_n(g)||_o}{||g||_o}\right|\le K',$$ which is the desired inequality.   
\end{proof}

\subsection{Central limit theorem}

We now complete the proof of the central limit theorem for the variables $\log||\Phi(g)||$ with $g\in\mathcal{P}_N$.

\begin{theo}\label{tcl}
Let $\mu$ be a nonelementary probability measure on $\text{Out}(F_N)$ with finite second moment with respect to $d_{CV_N}$. Let $\lambda>0$ be the drift of the random walk on $(\text{Out}(F_N),\mu)$ with respect to $d_{CV_N}$. Then there exists a centered Gaussian law $N_{\mu}$ on $\mathbb{R}$ such that for every compactly supported continuous function $F$ on $\mathbb{R}$, and all primitive elements $g\in \mathcal{P}_N$, one has $$\lim_{n\to +\infty}\int_{\text{Out}(F_N)} F \left(\frac{\log ||\Phi(g)||-n\lambda}{\sqrt{n}}\right)d\mu^{\ast n}(\Phi) = \int_{\mathbb{R}}F(t)dN_{\mu}(t),$$ uniformly in $g$.
\end{theo}

\begin{proof}
In view of Propositions \ref{dual} and \ref{norm-busemann}, it is enough to show that there exist $x\in\mathcal{UE}$ and a centered Gaussian law $N_{\mu}$ on $\mathbb{R}$ such that for every compactly supported continuous function $F$ on $\mathbb{R}$, one has $$\lim_{n\to +\infty}\int_{\text{Mod}(S)} F \left(\frac{\beta^-(\Phi,x)-n\lambda}{\sqrt{n}}\right)d\mu^{\ast n}(\Phi) = \int_{\mathbb{R}}F(t)dN_{\mu}(t).$$ This follows from Theorem \ref{tcl-gen} applied to $Y^-:=\mathcal{UE}$ and to $Y^+:=\partial_h^+CV_N$ (Hypothesis \textbf{(H1)} is checked in Corollary \ref{average}, applied to $\check{\mu}$, and Hypothesis \textbf{(H2)} is checked in Proposition \ref{product-estimate}).
\end{proof}

\bibliographystyle{amsplain}
\bibliography{bib-tcl}

\end{document}